\documentclass[final, 3p, times]{elsarticle}

\usepackage{amssymb}
\usepackage{amsmath}
\usepackage{amsthm}

\usepackage{graphicx,color,array}
\usepackage{comment}
\usepackage{lipsum,subeqnarray}
\usepackage{amsfonts}
\usepackage{graphicx}
\usepackage{epstopdf}
\usepackage{multirow} 
\usepackage{float}
\usepackage{latexsym}
\usepackage{amsmath,bm}
\usepackage{graphicx}
\usepackage{tikz}
\usepackage{amsthm}
\usetikzlibrary{positioning,arrows.meta}
\usepackage{subfig}
\usepackage{comment}
\usepackage[utf8]{inputenc}
\usepackage{ragged2e}
\usepackage{amsmath,amssymb}
\usepackage{algorithm}
\usepackage{algorithmic}
\usepackage{placeins}
\usepackage{array}
\usepackage{adjustbox}
\usepackage[utf8]{inputenc}
\newcommand{\tenq}[1]{\hbox{\oalign{${#1}$\crcr\hidewidth$\scriptscriptstyle\bm{\approx}$\hidewidth}}}

\newcommand{\Reals}[1]{{\rm I\! R}^{#1}}
\usepackage{subeqnarray}
\usepackage{graphicx}

\DeclareMathOperator*{\argmax}{arg\,max}

\newtheorem{lemma}{Lemma}
\newtheorem{theorem}{Theorem} 
\newtheorem{remark}{Remark} 
\newcommand{\tend}[1]{\hbox{\oalign{${#1}$\crcr\hidewidth$\scriptscriptstyle\bm{\sim}$\hidewidth}}}

\usepackage{lineno}
\usepackage{hyperref}
\usepackage{xcolor}
\usepackage{booktabs}

\journal{Journal of Computational Physics}

\begin{document}

\begin{frontmatter}



\title{Orthogonal greedy algorithm for linear operator learning with shallow neural network}

\author[JLU]{Ye Lin}
\ead{linye21@mails.jlu.edu.cn}
\affiliation[JLU]{
organization={School of Mathematics, Jilin University}, 
addressline={2699 Qianjin Street}, 
city={Changchun},
postcode={130012}, 
state={Jilin},
country={China}.}

\author[JLU]{Jiwei Jia\corref{*}}
\ead{jiajiwei@jlu.edu.cn}
\cortext[*]{Corresponding Author.}

\author[TXST]{Young Ju Lee}
\ead{yjlee@txstate.edu}
\affiliation[TXST]{
organization={Department of Mathematics, Texas State University}, 
addressline={601 University Drive}, 
city={San Marcos},
postcode={78666}, 
state={Texas},
country={U.S.A.}.}

\author[JLU]{Ran Zhang}
\ead{zhangran@jlu.edu.cn}


\begin{abstract}
Greedy algorithms, particularly the orthogonal greedy algorithm (OGA), have proven effective in training shallow neural networks for fitting functions and solving partial differential equations (PDEs). In this paper, we extend the application of OGA to the tasks of linear operator learning, which is equivalent to learning the kernel function through integral transforms. Firstly, a novel greedy algorithm is developed for kernel estimation rate in a new semi-inner product, which can be utilized to approximate the Green's function of linear PDEs from data. Secondly, we introduce the OGA for point-wise kernel estimation to further improve the approximation rate, achieving orders of accuracy improvement across various tasks and baseline models. In addition, we provide a theoretical analysis on the kernel estimation problem and the optimal approximation rates for both algorithms, establishing their efficacy and potential for future applications in PDEs and operator learning tasks.
\end{abstract}




\end{frontmatter}

\section{Introduction}

In recent years, deep neural networks have emerged as a powerful tool for solving partial differential equations (PDEs) in a wide range of scientific and engineering domains \cite{review-karniadakis2021physics}. Approaches in this area can be broadly classified into two main categories: (1) single PDE solvers and (2) operator learning. Single PDE solvers, such as physics-informed neural networks(PINNs)\cite{pinn-raissi2019physics}, the deep Galerkin method\cite{dgm-sirignano2018dgm}, the deep Ritz method\cite{dritz-yu2018deep}, optimize the deep neural network by minimizing a given loss function related to the PDE. These methods are specifically designed to solve a given instance of the PDE. In contrast, operator learning involves using deep neural networks to learn operators between function spaces, allowing for the learning of solution operators of PDEs from data pairs. 

Recently, several operator learning methods have been proposed, including deep Green networks (DGN)\cite{green-boulle2022gl}, deep operator networks (DON)\cite{don-lu2021learning}, and neural operators (NOs)\cite{no-li2020fourier}. While the underlying inspirations for these methods differ, they all establish a connection between operator learning and integral transforms. For linear PDEs, the solution operator can be represented as an integral transform involving a Green's function, typically unknown. Consequently, the task of learning a linear operator becomes equivalent to estimating a kernel function. Modeling this unknown function using a neural network and learning it from data is a straightforward approach, which has been realized in the work of GreenLearning networks (GL)\cite{green-boulle2022gl}. A primary challenge in operator learning is the optimization of deep neural network. Gradient-based methods, such as the Adam optimizer\cite{opt-kingma2014adam} or the Broyden–Fletcher–Goldfarb–Shanno (BFGS) algorithm\cite{opt-fletcher2000practical}, are commonly employed. However, the training procedure is prone to getting trapped in local minimum\cite{greentheory-boulle2023ellipticefficient}, which can significantly deviate from the theoretical approximation rate.

Although deep neural networks have achieved great success, considerable work has also been devoted to using shallow neural networks for solving PDEs, such as random feature methods\cite{rfm-chen2022rfm, rfm-Lamperski2024ApproximationWR, rfm-Rahimi2007RandomFF, rfm-Weinan2019ACA, rfm-zhang2024transferable}, extreme learning machines \cite{em-DONG2021114129, em-HUANG2006489, em-Lee2024AND}, and finite neuron methods \cite{fnm-xu2020fnm, fnm-park2022fnm}. Notably, greedy algorithm, especially orthogonal greedy algorithm(OGA) \cite{oga-barron2008approximation, oga-pati1993OMP, oga-barron1993universal, woga-temlyakov2000weak}, has been demonstrated to achieve optimal approximation rates \cite{oga-siegel2023greedy, oga-siegel2024entropy, oga-park2024randomized, ogatheory-siegel2024sharp} in shallow neural network optimization. To the best of our knowledge, shallow neural network methods are primarily applied to solve specific instances of target PDEs, rather than learning solution operators of PDEs. 

To avoid the optimization challenges of deep neural networks and to expand the application of shallow neural networks in operator learning, we design two novel OGA frameworks for shallow neural network optimization specifically for linear operator learning or kernel estimation tasks. A new semi-inner product is defined using finite data and kernel integrals, which is the key to establish the new OGA. We also provide a theoretical convergence rate of OGA in this context, which is also valid in the semi-norm. Our theory indicates that the larger the null space for the operator that corresponds to the semi-norm, the faster the convergence rate is expected. To further tackle the practical issue of shallow neural network approximating functions with sharp transitions or rapid oscillations and to enhance the approximation rate, we learn the kernel function point-wisely by OGA. The point-wise kernel estimation approximates a set of $d$-dimensional functions rather than $2d$-dimensional kernel function, leading to a faster optimal approximation rate. To validate the effectiveness of the proposed method, we consider multiple operator learning problems for linear PDEs and kernel estimation problems, including cases in 1D to 3D, as well as problems defined on irregular domains. Compared to baseline methods, the proposed approach achieves an order-of-magnitude improvement in accuracy.

Our main contributions can be summarized as follows:

\begin{itemize}
\item[(1)] We develop a novel OGA in a new semi-inner product for the shallow neural network optimization, which guarantees optimal approximation rates theoretically for linear operator learning and kernel estimation tasks.
\item[(2)] We introduce another new OGA framework for shallow neural network optimziation in point-wise kernel estimation tasks, referred to as PW-OGA, which further improves the convergence rate compared to directly kernel estimation using OGA.
\item[(3)] We validate both OGA and PW-OGA methods by learning Green's functions from various PDEs or kernel integrals. Numerical results demonstrate that both methods achieve the theoretical optimal approximation rate and outperform basline models by orders of magnitude in terms of accuracy in most cases.
\end{itemize}

Throughout the paper, we use $\tend{f}$ and $\tenq{G}$ to denote the vector and matrix, respectively. For some domain $\Omega \in \mathbb{R}^{d}$, the notation $L^2(\Omega)$ denotes the space of square integral function on $\Omega$, i.e., the $L^2$ function space with $L^2$ norm, $\|\cdot\|_{L^2(\Omega)}$. For $A, B \in \mathbb{R}$, we write $A \lesssim B$ if there is a generic constant $c$ such that $A \leq c B$. Lastly, we denote the Euclidean inner product in $\mathbb{R}^d$ by $\ell^2$.

The rest of the paper is organized as follows. In Section \ref{sec:ol}, we describe the problem settings of operator learning and discuss its connection to kernel estimation. In Section \ref{sec:snn}, we review the theory of nonlinear dictionary approximation with shallow neural network and present the OGA and PW-OGA methodology in detail. In Section \ref{sec:result}, we present several numerical results in which optimal approximation rates are verified, and we also compared the results with several state-of-the-art operator learning methods to demonstrate the effectiveness of our methods. Finally, in Section \ref{sec:conclusion}, we discuss the conclusions of our work.

\section{Operator learning and kernel estimation}
\label{sec:ol}
\subsection{Operator learning}
Consider Banach spaces $\mathcal{U}$ and $\mathcal{V}$, which consists of functions defined on a $d$-dimensional spatial domain $\Omega \subset \mathbb{R}^d$. Let $\mathcal{A} : \mathcal{U} \rightarrow \mathcal{V}$ be an operator mapping functions from $\mathcal{U}$ to $\mathcal{V}$. Given a set of observation function pairs $\{f_j, u_j\}^{N}_{j=1}$, where $f_j \in \mathcal{U}$ and $u_j \in \mathcal{V}$, the goal of operator learning is to learn a model $\mathcal{A}_{\theta}$ with finite parameters $\theta \in \mathbb{R}^p$ that approximates $\mathcal{A}$ by minimizing:
\begin{equation}
    \label{eq:operator_learning}
    \min_{\theta} \frac{1}{N} \sum_{j} \mathcal{C} (\mathcal{A}_{\theta}(f_j), u_j),
\end{equation}
where $\mathcal{C}: \mathcal{V} \times \mathcal{V} \rightarrow \mathbb{R}$ is a cost functional, such as relative $L^2(\Omega)$ norm:
\begin{equation}
    \label{eq:operator_learning_cost}
\mathcal{C} (\mathcal{A}_{\theta}(f(\mathbf{x})), u(\mathbf{x})) := \frac{\| u(\mathbf{x}) - \mathcal{A}_{\theta}(f(\mathbf{x})) \|_{L^2(\Omega)}}{\| u(\mathbf{x}) \|_{L^2(\Omega)}}. 
\end{equation}

Several operator learning methods have been recently proposed to solve PDEs, such as deep operator network (DeepONet) \cite{don-jin2022mionet, don-lu2021learning, don-wang2021learning, don-wang2022improved, don-gao2024adaptive, don-guo2024ib}, neural operators(NOs) \cite{no-kovachki2023neural, no-li2020fourier, no-li2020multipole, no-li2020gno, no-li2023fourier, no-rahman2023uno, no-tran2023factorized, no-you2022learning}, transformer based methods \cite{trans-cao2021choose, trans-guo2022transformer, trans-hao2023gnot, trans-kissas2022learning, trans-lee2022meshindependent, trans-li2024scalable} and multi-grid based methods \cite{mg-he2019mgnet, mg-he2024mgno, mg-zhu2023enhanced, mg-zhu2023fv}. These method employs deep neural networks(DNN) to learn the actions of target operators. 

\subsection{Linear operator learning and kernel estimation}
Under the assumption that the underlying differential operator $\mathcal{L}$ is a linear boundary value problem of the form, 
\begin{equation} 
\label{eq:linear_pde}
\begin{split}
\mathcal{L}(u)(\mathbf{x}) &= f(\mathbf{x}), \quad \mathbf{x} \in \Omega, \\
u(\mathbf{x}) &= 0, \quad\quad \mathbf{x} \in \partial\Omega,    
\end{split}
\end{equation}
the corresponding solution operator $\mathcal{A}$ can be written as the integral operators with Green's function function $G(\mathbf{x}, \mathbf{y})$ as 
\begin{equation} 
\label{eq:green_func}
\mathcal{A}(f)(\mathbf{x}) = u(\mathbf{x}) = \int_{\Omega} G(\mathbf{x},\mathbf{y}) f(\mathbf{y}) \mathrm{d} \mathbf{y}, \quad \mathbf{x},\mathbf{y} \in \Omega.
\end{equation}
In this setting, the linear operator learning task is equivalent to estimating the kernel function from data. For simplicity, we use the terms Green's function, kernel function and kernel interchangeably.

A key advantage of this approach is that it transforms the operator learning problem into a function approximation problem, which is more interpretable. Another benefit is that, after the kernel has been learned, the inference step only requires numerical kernel integral estimation, avoiding the need for a computationally intensive neural network forward pass. Several methods have been developed to approximate Green's function by neural network\cite{green-boulle2022gl, green-zhang2022modnet, green-gin2021deepgreen, green-lin2023bi, green-ji2023deep, green-peng2023deep, green-teng2022learning, green-lin2024green}. One of representative work is GreenLearning(GL)\cite{green-boulle2022gl}, which learns the empirical Green's function $G_{\theta}$ from data pairs by minimizing:
\begin{equation}
\label{eq:gl_loss}
\frac{1}{N} \sum_j \frac{\left \| u_j(\mathbf{x}) - \int_\Omega G_{\theta}( \mathbf{x}, \mathbf{y}) f_j( \mathbf{y}) \mathrm{d} \mathbf{y} \right \|_{L^2(\Omega)}}{\|u_j(\mathbf{x})\|_{L^2(\Omega)}},
\end{equation}
where $G_{\theta}$ is the neural network surrogate model of Green's function. Figure (\ref{fig:gl-arch}) presents the two main steps of GL: kernel function evaluation by a neural network and kernel integral estimation. 

\begin{figure}[h]
\centering
\includegraphics[scale=.7]{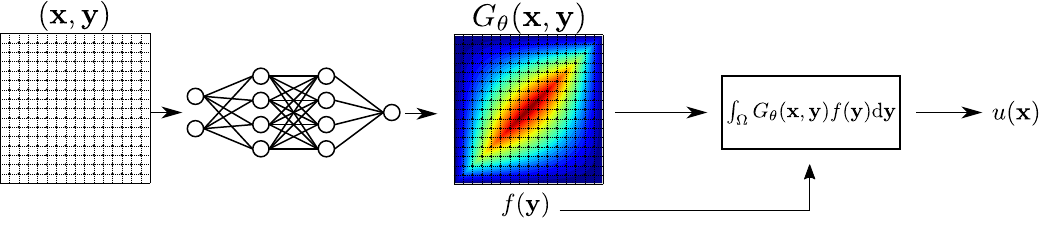}
\caption{Schematic of Green Learning}
\label{fig:gl-arch}
\end{figure}

Like many neural network-based operator learning methods, the approximation errors are mainly arise from numerical integration and optimization\cite{greentheory-boulle2023ellipticefficient}, which prevent achieving the theoretical order of accuracy. Currently, the most widely used approach for training neural networks is stochastic gradient descent(SGD) or its variant, such as ADAM\cite{opt-kingma2014adam}, often combined with a learning rate scheduler. However, these methods can often get stuck in local minima rather than converging to the global minimum. While these stochastic optimization methods can yield acceptable accuracy for many problems, training deep neural networks remains computationally expensive and challenging for error analysis\cite{oga-siegel2023greedy}. In contrast, the optimal approximation rate for shallow neural network can be achieved using the orthogonal greedy algorithm(OGA), with rigorous theoretical guarantees\cite{oga-siegel2023greedy, oga-siegel2024entropy, oga-park2024randomized}. To mitigate the optimization difficulties and extend the applicability of shallow neural network, we propose two new OGA-based methods, building on \cite{oga-park2024randomized}, to optimize shallow neural network for linear operator learning and kernel estimation tasks. These methods theoretically guarantee optimal approximation rates and demonstrate superior performance in numerical experiments.

\section{Nonlinear dictionary approximation with shallow neural network}
\label{sec:snn}
\subsection{Nonlinear dictionary approximation}

Let $H$ be a Hilbert space and $\mathbb{D} \subset H$ a dictionary of basis functions with a norm $\|\cdot\|_H$. The nonlinear dictionary approximation of a target function $G$ is achieved by expressing it as a sparse linear combination of $n$-term dictionary elements, as 
\begin{equation}
\label{eq:expansion}
G_n = \sum_{i=1}^{n} \alpha_i g_i,
\end{equation}
where $g_i \in \mathbb{D}$ depend upon the function $G$. Following \cite{ogatheory-siegel2024entropy, ogatheory-siegel2024sharp}, we assume that the target function $G$ lies in the convex hull of the dictionary $\mathbb{D}$. The closed convex hull of $\mathbb{D}$ can be written as 
\begin{equation}
\label{eq:convex_hull}
\begin{split}
{\rm co}(\mathbb{D}) = \overline{ \Biggl\{  \sum_{i=1}^{n} \alpha_i g_i : n \in \mathbb{N}, g_i \in \mathbb{D}, \sum_{i=1}^n |\alpha_i| \leq 1 \Biggl\} }, 
\end{split}
\end{equation}
and the gauge norm $\| \cdot \|_{\mathcal{L}_1(\mathbb{D})}$ of ${\rm co}{(\mathbb{D})}$ denotes,
\begin{equation}
\label{eq:gauge_norm}
\begin{split}
\|v \|_{\mathcal{L}_1(\mathbb{D})} = \inf \left \{ \sum_{i}|c_i| : v = \sum_{g_i \in \mathbb{D}} c_i g_i \right \}, 
\end{split}
\end{equation}
where $\mathcal{L}_1(\mathbb{D})$ is the subspace of $H$ defined by the variation norm,
\begin{equation}
\label{eq:var_norm}
\begin{split}
\mathcal{L}_1(\mathbb{D}) := \inf \{ G \in H : \|G\|_{\mathcal{L}_1(\mathbb{D})} < \infty \}.
\end{split}
\end{equation}
With $\|\mathbb{D}\| = \sup_{g \in \mathbb{D}} \|g\|_H \leq 1$, given a function $G \in \mathcal{L}_1(\mathbb{D})$, \cite{ogatheory-pisier1981remarques} gives a classical result that there exists an $n$-term dictionary elements which satisfies
\begin{equation}
\label{eq:approx_rate}
\begin{split}
\inf_{G_n \in \Sigma_{n}(\mathbb{D})} \| G - G_n \|_{H} \lesssim \|G\|_{\mathcal{L}_1(\mathbb{D})} n^{-\frac{1}{2}},
\end{split}
\end{equation}
where 
\begin{equation}
\label{eq:model_class}
\begin{split}
\Sigma_{n}(\mathbb{D}) &= \Biggl\{ \sum_{i=1}^{n} \alpha_i g_i : \alpha_i \in \mathbb{R}, g_i \in \mathbb{D} \Biggl\}, 
\end{split}
\end{equation}
is the set of $n$-term function class.

\subsection{Shallow neural network for function approximation in a Hilbert space}

Typically, we may choose shallow neural networks to construct dictionary $\mathbb{D}$. A shallow neural network with input $\mathbf{x} \in \mathbb{R}^d$ is given by 
\begin{equation}
\label{eq:shallownet}
G_n(\mathbf{x}) = \sum_{i=1}^{n} \alpha_i \sigma ( \mathbf{w}_i \cdot \mathbf{x} + \beta_i),
\end{equation}
where $\sigma$ is an activation function, $n$ is the number of neurons, $\{\alpha_i\}_{i=1}^n \subset \mathbb{R}$, $\{\mathbf{w}_i\}_{i=1}^n \subset S^{d-1}$, and $\{\beta_i\}_{i=1}^n \subset \mathbb{R}$ are learnable parameters, and $S^{d-1}$ represents the unit $d$-sphere in $\mathbb{R}^d$. For $\mathrm{ReLU}^k$ activation function $\sigma_k(t) = \max(0,t)^k$, where $t \in \mathbb{R}$, following \cite{oga-siegel2023greedy, oga-park2024randomized, ogatheory-siegel2024sharp}, we define shallow neural network function class $\Sigma_{n,M}(\mathbb{D})$ for some $M>0$, as 
\begin{equation}
\label{eq:shallow_dict}
\begin{split}
\Sigma_{n, M}(\mathbb{D}) &= \Biggl\{ \sum_{i=1}^{n} \alpha_i g_i : \alpha_i \in \mathbb{R}, g_i \in \mathbb{D}, \sum_{i=1}^n |\alpha_i| \leq M \Biggl\}, \\
\end{split}
\end{equation}
where dictionary $\mathbb{D}$ is a collection of $\mathrm{ReLU}^k$ shallow neural networks,
\begin{equation}
\label{eq:shallow_nn}
\begin{split}
\mathbb{D} &:= \bigl\{ \pm \sigma_{k}(\mathbf{w} \cdot \mathbf{x} + \beta) : \mathbf{w} \in S^{d-1}, \beta \in [c_1, c_2] \bigl\},
\end{split}
\end{equation}
with two constants $c_1$ and $c_2$ are chosen to satisfy
\begin{equation}
\label{eq:cbound}
\begin{split}
c_1 \leq \inf\{\mathbf{w} \cdot \mathbf{x}: \mathbf{x} \in \Omega, \mathbf{w} \in S^{d-1}\} \leq \sup \{\mathbf{w} \cdot \mathbf{x}: \mathbf{x} \in \Omega, \mathbf{w} \in S^{d-1}\} \leq c_2.
\end{split}
\end{equation}
A sharp bound on the approximation rate of $\Sigma_{n, M}(\mathbb{D})$ was developed in \cite{ogatheory-siegel2024sharp}, especially for $H = L^2(\Omega)$,
\begin{equation}
\label{eq:approx_rate_relu}
\begin{split}
\inf_{G_n \in \Sigma_{n,M}(\mathbb{D})} \| G - G_n \|_{H} \leq C \|G\|_{\mathcal{L}_1(\mathbb{D})} n^{-\frac{1}{2} - \frac{2k+1}{2d}},
\end{split}
\end{equation}
where $\Sigma_{n,M}(\mathbb{D})$ and $\mathcal{L}_1(\mathbb{D})$ were defined as (\ref{eq:shallow_dict}) and (\ref{eq:var_norm}), respectively, and $C$ is a constant independent of $G$ and $n$. 

\subsection{Inner product in the linear operator learning setting}\label{sec:green}

We consider a kernel function space consisting of $G(\mathbf{x}, \mathbf{y}) : \Omega \times \Omega \mapsto \mathbb{R}$ such that
\begin{equation}\label{prop} 
\sup_{\mathbf{x} \in \Omega} \int_{\Omega} |G(\mathbf{x},\mathbf{y})| \mathrm{d} \mathbf{y} < \infty \quad \mbox{ and } \quad \sup_{\mathbf{y} \in \Omega} \int_{\Omega} |G(\mathbf{x},\mathbf{y})| \mathrm{d}\mathbf{x} < \infty
\end{equation} 
so that it defines a bounded integral operator (see, e.g. [\cite{follandbook}, Theorem 6.18]). For a function $f(\mathbf{x})$, where $\mathbf{x}, \mathbf{y} \in \Omega \subset \mathbb{R}^d$, we denote the corresponding integral transform as: 
\begin{equation}
\label{eq:star}
G \star f := \int_\Omega G(\mathbf{x},\mathbf{y}) f(\mathbf{y}) \mathrm{d} \mathbf{y}.
\end{equation}
Our main objective is to learn the kernel function $G$ from a given dataset $\{f_j, u_j\}^{N}_{j=1}$. In practice, dataset are often obtained by sensors or sampled from continuous functions. Therefore, it is natural to consider $\{f_j, u_j\}^{N}_{j=1}$ as functions that are piecewise linearly interpolated from the discrete counterpart $\{\tend{f_j}, \tend{u_j}\}^{N}_{j=1}$, where $\tend{f_j} \in \mathbb{R}^{m_f}$ and $\tend{u_j} \in \mathbb{R}^{m_u}$ are vectors. We assume that two function samples $\tend{f_i}$ and $\tend{f_j}$ are linearly independent in $\mathbb{R}^{m_f}$, then $f_i$ and $f_j$ are also linearly independent as functions in $\mathcal{L}_1(\mathbb{D})$. For simplicity, we also assume that $f_j$ has a unit norm, and denote the set $\mathcal{K}$ by the following:
\begin{equation}
\label{eq:data_space}
\mathcal{K} = \left \{ f \in \mathcal{L}_1(\mathbb{D}) : f = \sum_{j=1}^N \alpha_j f_j \quad \mbox{ with } \quad \sum_{j=1}^N |\alpha_j| \leq 1 \quad \mbox{ and } \quad \|f_j\|_{L^2(\Omega)} = 1 \right \}. 
\end{equation}
Based on \eqref{eq:star}, we define a new inner product as follows:
\begin{equation}\label{eq:kernel_inner_prod}
\langle G_1, G_2 \rangle_H := \frac{1}{N} \sum_{j=1}^N (G_1 \star f_j , G_2 \star f_j)_{L^2(\Omega)}. 
\end{equation} It is evident that unless $N$ is sufficiently large, $\langle \cdot, \cdot \rangle_H$ does not make a definite inner product, but rather a semi-inner product. Namely, the induced norm, denoted by $\|\cdot\|_H$ is only a semi-norm. More precisely, we have:   
\begin{equation}
0 = \|G\|_H^2 = \langle G, G \rangle_H = \frac{1}{N} \sum_{j=1}^N (G \star f_j, G\star f_j)_{L^2(\Omega)} \geq 0,
\end{equation}
but this does not necessarily imply that
$G = 0$. This can be clearly observed in the extreme case for $N = 1$ and $f_1 = 1$, if $G(\mathbf{x},\mathbf{y}) \neq 0$ can still satisfy $\int_\Omega G(\mathbf{x},\mathbf{y}) \mathrm{d} \mathbf{y} = 0$. In fact, the linear operator learning is inherently an approximation for the kernel function using a limited number of data points. Therefore, the kernel must be expressed in some approximate form, denoted as $\hat{G}$, to make it trainable. Similarly, we may treat $\hat{G}$ as a piecewise linear interpolated from the discrete counterpart $\tenq{G} \in \mathbb{R}^{m_u \times m_f}$ of $G$. Specifically, we have 
\begin{equation}
\mathcal{G} := \left \{ \hat{G} \in \mathcal{L}_1(\mathbb{D}) : \exists \tenq{G} \in \mathbb{R}^{m_u \times m_f} \mbox{ such that } \tenq{G} \tend{f} = \begin{pmatrix} 
(G \star f)(\mathbf{x}_1)  \\ 
\vdots \\
(G \star f)(\mathbf{x}_{m_u}) 
\end{pmatrix} \right \}. 
\end{equation}


We give an example of evaluation of a one-dimensional integral transform for illustration. Assume that $y_k = y_0 + kh$ are equidistant grid points on $[-1,1]$, where $k$ is the index and $h$ is the mesh size. The discrete approximation of a function on this grid is denoted by the discrete vector $\tend{f} = [f^h_0, \cdots, f^h_k, \cdots, f^h_{m_f}]$, and the piecewise constant approximation denotes as $f(y) = f^h_k, y \in [y_k - h/2, y_k + h/2]$. The kernel integral can be calculated as,

\begin{equation}
\label{eq:log_int}
\begin{split}
u(x) &= (G \star f)(x) = \int_{-1}^{1} G(x,y) f(y) \mathrm{d}y \\
&= \sum_k \left ( \int_{y_k-h/2}^{y_k+h/2} G(x,y_k) f(y) \mathrm{d}y \right ) \\
&= \sum_k \left ( \int_{y_k-h/2}^{y_k+h/2} G(x,y_k) \mathrm{d}y \right ) f^h_k = \sum_k \hat{G}(x,y_k) f^h_k \\
\end{split}
\end{equation}
Given $G(x,y) = \ln|x-y|$, the corresponding $\hat{G}(x,y_k)$ can be written analytically\cite{mlmi-hernandez2005acoustics},
\begin{equation}
\label{eq:coef-case}
\hat{G}(x,y_k) = \begin{cases}
\left ( x-y_k+\frac{h}{2} \right )\ln \left ( x-y_k+\frac{h}{2} \right ) - \left ( x-y_k-\frac{h}{2} \right )\ln \left ( x-y_k-\frac{h}{2} \right ) - h, & x > y_k, \\
h\ln \left ( \frac{h}{2} \right )-h, & x = y_k, \\
\left ( y_j-x+\frac{h}{2} \right )\ln \left ( y_k-x+\frac{h}{2} \right ) - \left ( y_j-x-\frac{h}{2} \right )\ln \left ( y_k-x-\frac{h}{2} \right ) - h, & x < y_k
  \end{cases}
\end{equation}
Figure (\ref{fig:bounded_approx}) shows the difference between $G$ and $\hat{G}$. The exact kernel $G(x,y) = \ln|x-y|$ is an unbounded function with singularities at $x=y$, while $\hat{G}$ approximates $G$ and removes the singularity. To achieve a better approximation, we can choose a finer grid with smaller $h$.

\begin{figure}[h]
\centering
\includegraphics[scale=.7]{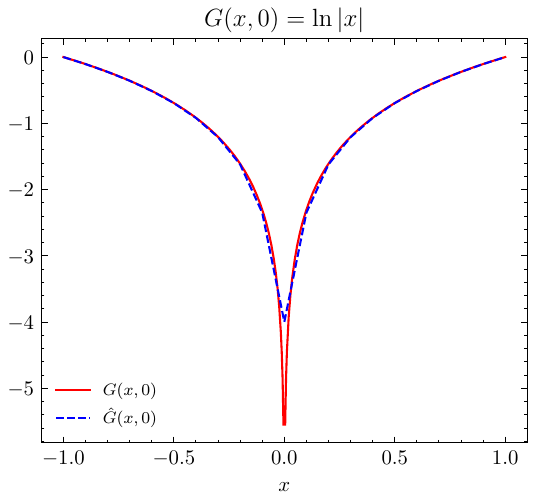}
\caption{exact kernel $G$ and its approximant $\hat{G}$ with $h=0.1$}
\label{fig:bounded_approx}
\end{figure}

Let $G_n$ denote the $n$-neuron approximation of $G$. Then, we have the following inequality: 
\begin{equation}
\label{eq:ieq}
\|G - G_n\|_H \leq \|G - \hat{G}\|_H + \|\hat{G} - G_n\|_H.
\end{equation}
(\ref{eq:ieq}) shows that the error in linear operator learning with shallow neural network consists of two components. The first is the modeling error, $\|G - \hat{G}\|_H$, which represents the difference between the true kernel $G$ and $\hat{G}$. The second component is the approximation error, $\|\hat{G} - G_n\|_H$, which measures the discrepancy between the $\hat{G}$ and $G_n$. In operator learning setting, it is typical to ignore the modeling error and to focus on the second part of the optimization. However, it is important to ensure that $\hat{G}$ is constructed in such a way that it serves as a good approximation of $G$, while also being sufficiently trainable with the limited available data. For simplicity, from now on, we directly use $G$ as $\hat{G}$ in the following derivation and discussion, under the assumption that the modeling error is neglectable.

Before we move to the orthogonal greedy algorithm(OGA), we need to discuss the null space for the norm defined in 
\begin{equation}
\label{eq:null_space1}
\mathcal{N} = \{ G \in \mathcal{G} : G \star f = 0,~ \forall f \in \mathcal{K} \}.
\end{equation}
We shall see that the space $\mathcal{G}$ as a subspace in $L^2(\Omega\times \Omega)$ and consider the orthogonal complement of $\mathcal{N}^\perp$. The dimension of $\mathcal{N}^\perp$is expected to be quite small in case that the dimension of $\mathcal{K}$ is small, i.e., the data are more or less linearly dependent. We observe that such an approximate expression can still restrict the case that the operator learning will give a unique $G$ and it is crucially dependent on the available data. More precisely, if the data set $\{f_j, u_j\}^{N}_{j=1}$ does not contain sufficiently large linearly independent data sets, then the optimization process does not lead to a unique minimizer. It may give a rather different minimizer that fits the data well as the norm $\|\cdot\|_H$ is only seminorm. Thus, one has to establish the OGA approximation property in the semi-norm, which is a new contribution in this paper. Given $G \in \mathcal{G}$ and $f \in \mathcal{K}$ with a total of $m_f$ number of linear independent functions, we establish the following lemma:
\begin{lemma}
Assume that the data set $\{f_j\}^{N}_{j=1}$ contains at least $m_f$ linear independent functions, say $\{f_{j_\ell}\}^{m_f}_{\ell=1}$ with $j_\ell \in \{1,\cdots,N\}$. Then, the space $\mathcal{K}$ makes a finite dimensional compact set and the following bilinear form on $\mathcal{G}$ defines an inner product:
\begin{equation}
\langle G_1, G_2 \rangle_H = \frac{1}{N} \sum_{j=1}^N (G_1 \star f_j, G_2 \star f_j)_{L^2(\Omega)}, \quad \forall G_1, G_2 \in \mathcal{G}.
\end{equation}
\end{lemma}
\begin{proof}
We only need to show the positivity of the inner product. For $G \in \mathcal{G}$, since we have at least $m_f$ linear independent sets $\{f_{j_\ell}\}_{\ell=1}^{m_f}$, we have that
\begin{equation}
G \star f_{j_\ell} = 0, \quad\forall \ell = 1,\cdots,m_f.  
\end{equation}
This necessarily implies that  
\begin{equation}
G = 0. 
\end{equation}
This completes the proof. 
\end{proof}
In passing to the next section, we shall establish that the norm $\|\cdot\|_H$ is bounded by the quotient norm given as follows: 
\begin{lemma}\label{main:lmb}
The following holds: 
\begin{equation} 
\label{eq:G_H_norm_bound}
\|G\|_H \lesssim \|G\|_{L^2(\Omega \times \Omega)/\mathcal{N}} = \inf_{\psi \in \mathcal{N}} \|G + \psi\|_{L^2(\Omega\times \Omega)}.  
\end{equation} 
\end{lemma}
\begin{proof}
By the definition of $\|G\|_H$ and for any $\psi \in \mathcal{N}$, we have that  
\begin{eqnarray}
\|G\|_H^2 &=& \frac{1}{N} \sum_{j=1}^N (G \star f_j, G\star f_j)_{L^2(\Omega)} = \frac{1}{N} \sum_{j=1}^N ((G + \psi) \star f_j, (G + \psi)\star f_j)_{L^2(\Omega)} \\
&=& \frac{1}{N} \sum_{j=1}^N \int_\Omega \Big( \int_\Omega (G + \psi)(\mathbf{x},\mathbf{y}) f_j (\mathbf{y}) \, \mathrm{d}\mathbf{y}  \Big)^2 \mathrm{d} \mathbf{x} \\
&\leq& \frac{1}{N} \sum_{j=1}^N \int_\Omega \int_\Omega (G + \psi)^2(\mathbf{x},\mathbf{y}) \mathrm{d}\mathbf{y} \int_\Omega f_j^2 (\mathbf{y}) \, \mathrm{d}\mathbf{y} \mathrm{d}\mathbf{x} = \frac{1}{N} \sum_{j=1}^N \int_\Omega \int_\Omega (G + \psi)^2(\mathbf{x},\mathbf{y}) d\mathbf{y} \mathrm{d}\mathbf{x} \\
&=& \|G + \psi\|_{L^2(\Omega\times\Omega)}^2, 
\end{eqnarray}
where the last equality is due to the assumption that the data is normalized, i.e., $f_j$'s are such that $\|f_j\|_{L^2(\Omega)} = 1$. 
This completes the proof. 
\end{proof}

\subsection{Orthogonal greedy algorithm for function approximation}

A common class of algorithms to approximate target function $G$ through the expansion (\ref{eq:expansion}) are greedy algorithms, such as the pure greedy algorithm(PGA)\cite{oga-zhang1993pga}, the relaxed greedy algorithm(RGA)\cite{oga-jones1992simple}\cite{oga-barron1993universal}\cite{oga-barron2008approximation}, the orthogonal greedy algorithm(OGA)\cite{oga-pati1993OMP}, and their weak counterparts\cite{woga-temlyakov2000weak}. In this section, we introduce some notations, then present OGA with randomized discrete dictionaries\cite{oga-park2024randomized} on function approximation and its extension to kernel estimation. 

Given a target function $G$ in a Hilbert space with a inner product $( \cdot, \cdot )_{L^2(\Omega)}$ and the function norm $\|\cdot\|_{L^2(\Omega)}$, we consider approximating $G_n$ in $\|\cdot\|_{L^2(\Omega)}$ norm as (\ref{eq:expansion}) with a shallow neural network dictionary $\mathbb{D}$ as defined in (\ref{eq:shallow_dict}). The orthogonal greedy algorithm(OGA) for function approximation can be summarized as combining the greedy algorithm with orthogonal projections, also known as orthogonal matching pursuit given as the Algorithm \ref{OGA}, where $P_n$ is the orthogonal projection onto subspace $H_n := {\rm span} \{g_1,\cdots,g_n\}$.
\begin{algorithm}[h]
\caption{Orthogonal greedy algorithm for function approximation}\label{OGA}
Given a dictionary $\mathbb{D}$ defined in (\ref{eq:shallow_dict}), $G_0 = 0$
\begin{algorithmic}
\FOR{$n = 1, 2, \cdots $}
\STATE{
Subproblem optimization
\begin{equation}\label{eq:oga_subproblem}
g_n = \argmax_{g \in \mathbb{D}} (G - G_{n-1}, g )_{L^2(\Omega)}
\end{equation}}
\STATE{
Orthogonal projection
\begin{equation}\label{eq:oga_projection}
G_n = P_n(G)
\end{equation}
}
\ENDFOR
\end{algorithmic}
\end{algorithm}



One way to solve the argmax subproblem(\ref{eq:oga_subproblem}) is to approximate the solution by exhaustively searching over a deterministic discrete dictionary $\mathbb{D}$\cite{oga-siegel2023greedy}, however, it is computationally prohibited on high-dimensional problems. \cite{oga-park2024randomized} introduce a randomized discrete dictionary at each iteration of OGA, significantly reducing the size of the dictionaries compared with the deterministic one. 
 
Before discretizing the dictionary $\mathbb{D}$, we first introduce the parameter space of shallow neural networks. For $\mathbf{w} = (\omega_1, \cdots, \omega_d) \in S^{d-1}$, we have $\mathbf{w} = \mathcal{S}(\phi)$ for some $\phi = (\phi_1, \cdots, \phi_{d-1}) \in [0, \pi]^{d-2} \times [0, 2\pi) \subset \mathbb{R}^{d-1}$, where hyperspherical coordinate map $\mathcal{S}$ is given by \cite{oga-siegel2023greedy, oga-park2024randomized}:
\begin{equation}\label{eq:hypersphere_map}
\begin{split}
\omega_i = (\mathcal{S}(\phi))_i =
    \left\{
            \begin {aligned}
                & \Biggl( \prod_{j=1}^{i-1} \sin \phi_j \Biggl) \cos \phi_i, & 1 \leq i \leq d-1,\\
                & \prod_{j=1}^{d-1} \sin \phi_j, & i=d.
            \end{aligned}
    \right.
\end{split}
\end{equation}
The parameter space of shallow neural network $R$ is defined as $\bigl([0, \pi]^{d-2} \times [0, 2\pi) \bigl) \times [c_1, c_2]$. A random discretization of the dictionary $\mathbb{D}$ with $N_R$ samples from the uniform distribution on $R$ is defined as
\begin{equation}\label{eq:rand_dict}
\mathbb{D}_{N_R} = \bigl\{ \pm \sigma_{k}(\mathcal{S}(\phi_i) \cdot \mathbf{x} + \beta_i) : (\phi_i, \beta_i) \sim \mathrm{Uniform}(R), 1 \leq i \leq N_R\}
\end{equation}
The orthogonal greedy algorithm with randomized discrete dictionaries for function approximation is given as the Algorithm \ref{OGA_rand}.

\begin{algorithm}[h]
\caption{Orthogonal greedy algorithm for function approximation with randomized discrete dictionaries}\label{OGA_rand}
Given $G_0 = 0$
\begin{algorithmic}
\FOR{$n = 1, 2, \cdots $}
\STATE{
Random dictionary construction 
\begin{equation}
\mathbb{D}_{N_R} = \bigl\{ \pm \sigma_{k}(\mathcal{S}(\phi_i) \cdot \mathbf{x} + \beta_i) : \phi_i, \beta_i \sim \mathrm{Uniform}(R), 1 \leq i \leq N_R\}    
\end{equation}}
\STATE{
Subproblem optimization
\begin{equation}
g_n = \argmax_{g \in \mathbb{D}_{N_R}} (G - G_{n-1}, g)_{L^2(\Omega)}
\end{equation}}
\STATE{
Orthogonal projection
\begin{equation}
G_n = P_n(G)
\end{equation}
}
\ENDFOR
\end{algorithmic}
\end{algorithm}

In \cite{oga-park2024randomized}, OGA with randomized discrete dictionaries is proven to be a realization of a weak orthogonal greedy algorithm(WOGA), which replaces (\ref{eq:oga_subproblem}) with a weaker condition: finding $g_n \in \mathbb{D}$ such that $\langle G - G_{n-1}, g_n \rangle_H \geq \gamma \max_{g \in \mathbb{D}} \langle G - G_{n-1}, g \rangle_H$ for a given parameter $\gamma \in (0,1]$. The convergence rate of WOGA is given by,
\begin{equation}
\label{eq:woga_approx_rate}
\begin{split}
\inf_{G_n \in \Sigma_{n,M}(\mathbb{D})} \| G - G_n \|_{H} \leq \frac{C}{\gamma} \|G\|_{\mathcal{L}_1(\mathbb{D})} n^{-\frac{1}{2} - \frac{2k+1}{2d}}.
\end{split}
\end{equation}
For details on the convergence analysis of WOGA, see \cite{woga-temlyakov2000weak, oga-siegel2024entropy, oga-park2024randomized}.

\subsection{Orthogonal greedy algorithm for kernel estimation}
In this section, we now discuss how to obtain the shallow neural network approximation of kernel from dataset $\{f_j, u_j\}^{N}_{j=1}$ by OGA. Following the setting of function approximation, we consider approximating a target kernel $G$ by $G_n$ in $\|\cdot\|_H$ norm as expansion (\ref{eq:expansion}) through shallow neural network dictionary. Algorithm \ref{OGA_rand_kernel} shows the details. 

\begin{algorithm}[h]
\caption{Orthogonal greedy algorithm for kernel estimation with randomized discrete dictionaries}\label{OGA_rand_kernel}
Given a dataset $\{f_j, u_j\}^{N}_{j=1}$, $G_0 = 0$
\begin{algorithmic}
\FOR{$n = 1, 2, \cdots $}
\STATE{
Random dictionary construction 
\begin{equation}
\label{eq:rand_kernel_dictionary}
\mathbb{D}_{N_R} = \bigl\{ \pm \sigma_{k}(\mathcal{S}(\phi_i) \cdot [\mathbf{x}, \mathbf{y}] + \beta_i) : \phi_i, \beta_i \sim \mathrm{Uniform}(R), 1 \leq i \leq N_R\}    
\end{equation}}
\STATE{
Subproblem optimization
\begin{equation}
\label{eq:kernel_maximum}
g_n = \argmax_{g \in \mathbb{D}_{N_R}} \langle G - G_{n-1}, g \rangle_H    
\end{equation}}
\STATE{
Orthogonal projection
\begin{equation}
G_n = P_n(G)
\end{equation}
}
\ENDFOR
\end{algorithmic}
\end{algorithm}

One need to notice that, the target kernel $G$ is a function with $2d$ inputs: a neuron basis is $g_i = \sigma(\mathbf{w}_i \cdot [\mathbf{x}, \mathbf{y}] + \beta_i) \in \mathbb{D}_{N_R}$ (\ref{eq:rand_kernel_dictionary}) depends on the function $G$, where $\mathbf{x}, \mathbf{y} \in \Omega \subset \mathbb{R}^d$ are two points, and $[\mathbf{x}, \mathbf{y}] \in \Omega \times \Omega \subset \mathbb{R}^{2d}$ denotes concatenated points. The other change is to replace the function inner product $( \cdot, \cdot )_{L_2(\Omega)}$ with the inner product defined in (\ref{eq:kernel_inner_prod}).

In subproblem optimization step, the target function $\langle G - G_{n-1}, g \rangle_H$ can be written as, 
\begin{equation}
\begin{split}
\langle G - G_{n-1}, g \rangle_H &= \sum_{j=1}^{N} ((G - G_{n-1}) \star f_j ,  g \star f_j)_{L^2(\Omega)} \\
&= \sum_{j=1}^{N} ( G \star f_j - G_{n-1} \star f_j ,  g \star f_j)_{L^2(\Omega)} \\
&= \sum_{j=1}^{N} ( u_j - G_{n-1} \star f_j ,  g \star f_j)_{L^2(\Omega)} \\
&= \sum_{j=1}^{N} \int_{\Omega} (u_j(\mathbf{x}) - G_{n-1} \star f_j) (g \star f_j)  \mathrm{d} \mathbf{x}.
\end{split}
\end{equation}
The maximization problem (\ref{eq:kernel_maximum}) can be solved by an exhaustive search over $\mathbb{D}_{N_R}$ (\ref{eq:rand_kernel_dictionary}) or combined with other optimization methods, such as Newton's method.

After calculating $g_n$, we shall consider $G_n$ by $G_n = P_n(G)$. For $\forall v \in \mathrm{span} \{g_1, g_2, \cdots g_n \}$, we have $\langle P_n(G), v \rangle_H = \langle G, v \rangle_H$. By plug-in $G_n = \sum_{i=1}^{n} \alpha_i g_i$ and let $v = g_1$, we have
\begin{equation}
\label{eq:g1}
\begin{split}
\langle P_n(G), g_1 \rangle_H &= \langle G, g_1 \rangle_H \\
\sum_{j=1}^{N} \left (\sum_{i=1}^{n} \alpha_i g_i \star f_j ,  g_1 \star f_j \right )_{L^2(\Omega)} &= \sum_{j=1}^{N} \left (G \star f_j ,  g_1 \star f_j \right )_{L^2(\Omega)} \\ 
\sum_{j=1}^{N} \left (\sum_{i=1}^{n} \alpha_i g_i \star f_j ,  g_1 \star f_j \right )_{L^2(\Omega)} &= \sum_{j=1}^{N} (u_j ,  g_1 \star f_j)_{L^2(\Omega)} \\ 
    \end{split}
\end{equation}
For the sake of notational simplicity, we denote $\Hat{u}_{j,i} := g_i \star f_j$, then (\ref{eq:g1}) becomes
\begin{equation}
\label{eq:e22}
\begin{split}        
\sum_{j=1}^{N} \left (\sum_{i=1}^{n} \alpha_i \Hat{u}_{j,i} ,  \Hat{u}_{j,1} \right )_{L^2(\Omega)} &= \sum_{j=1}^{N} (u_j, \Hat{u}_{j,1})_{L^2(\Omega)}. \\ 
    \end{split}
\end{equation}
(\ref{eq:e22}) can be rewritten as a linear combination as follows, 
\begin{equation}
    \begin{split}        
        \sum_{j=1}^{N} (\alpha_1 \Hat{u}_{j,1} ,  \Hat{u}_{j,1})_{L^2(\Omega)} + \sum_{j=1}^{N} (\alpha_2 \Hat{u}_{j,2} ,  \Hat{u}_{j,1})_{L^2(\Omega)} + \cdots + \sum_{j=1}^{N} (\alpha_n \Hat{u}_{j,n} ,  \Hat{u}_{j,1})_{L^2(\Omega)} &= \sum_{j=1}^{N} (u_j ,  \Hat{u}_{j,1})_{L^2(\Omega)} \\ 
        \alpha_1 \sum_{j=1}^{N} (\Hat{u}_{j,1} ,  \Hat{u}_{j,1})_{L^2(\Omega)} + \alpha_2 \sum_{j=1}^{N} ( \Hat{u}_{j,2} ,  \Hat{u}_{j,1})_{L^2(\Omega)} + \cdots + \alpha_n \sum_{j=1}^{N} ( \Hat{u}_{j,n} ,  \Hat{u}_{j,1})_{L^2(\Omega)} &= \sum_{j=1}^{N} (u_j ,  \Hat{u}_{j,1})_{L^2(\Omega)} \\ 
    \end{split}
\end{equation}
After rewriting $\langle P_n(G), g_2 \rangle_H = \langle G, g_2 \rangle_H, \cdots, 
        \langle P_n(G), g_n \rangle_H = \langle G, g_n \rangle_H$ as above, we obtain the following linear system, 
\begin{equation}
    \label{eq:linear_system}
    \begin{bmatrix}
    \sum_{j=1}^{N} (\Hat{u}_{j,1} ,  \Hat{u}_{j,1})_{L^2(\Omega)} & \sum_{j=1}^{N} ( \Hat{u}_{j,2} ,  \Hat{u}_{j,1})_{L^2(\Omega)} & \cdots & \sum_{j=1}^{N} ( \Hat{u}_{j,n} ,  \Hat{u}_{j,1})_{L^2(\Omega)} \\
    \sum_{j=1}^{N} (\Hat{u}_{j,1} ,  \Hat{u}_{j,2})_{L^2(\Omega)} & \sum_{j=1}^{N} ( \Hat{u}_{j,2} ,  \Hat{u}_{j,2})_{L^2(\Omega)} & \cdots & \sum_{j=1}^{N} ( \Hat{u}_{j,n} ,  \Hat{u}_{j,2})_{L^2(\Omega)} \\
    \vdots & \vdots & \ddots & \vdots \\
    \sum_{j=1}^{N} (\Hat{u}_{j,1} ,  \Hat{u}_{j,n})_{L^2(\Omega)} & \sum_{j=1}^{N} ( \Hat{u}_{j,2} ,  \Hat{u}_{j,n})_{L^2(\Omega)} & \cdots & \sum_{j=1}^{N} ( \Hat{u}_{j,n} ,  \Hat{u}_{j,n})_{L^2(\Omega)} \\
    \end{bmatrix}
    \begin{bmatrix}
    \alpha_1 \\
    \alpha_2 \\
    \vdots \\
    \alpha_n
    \end{bmatrix} = 
    \begin{bmatrix}
    \sum_{j=1}^{N} (u_{j} ,  \Hat{u}_{j,1})_{L^2(\Omega)} \\
    \sum_{j=1}^{N} (u_{j} ,  \Hat{u}_{j,2})_{L^2(\Omega)} \\
    \vdots \\
    \sum_{j=1}^{N} (u_{j} ,  \Hat{u}_{j,n})_{L^2(\Omega)} \\
    \end{bmatrix}.
\end{equation}
The linear combination coefficients $\{\alpha_i\}_{i=1}^n$ can be calculated by solving linear system (\ref{eq:linear_system}).

\subsection{Approximation rate analysis of OGA for kernel estimation}
We shall now discuss the approximation rate for the weak orthogonal greedy algorithm (\ref{eq:woga_approx_rate}). We first introduce the notion of metric entropy introduced by Kolmogorov \cite{Tikhomirov1993}, the number $\varepsilon(A)$ of a set
$A \subset H$ are defined by
\begin{equation}
\varepsilon_n(A)_H = \inf \left \{ \varepsilon > 0 : A \mbox{ is covered by } 2^n \mbox{ balls of radius } \varepsilon \right \}.
\end{equation}
We now introduce an important but simple lemma, established in elsewhere \cite{9698049,li2023entropy}.  
to obtain the main result. 
\begin{lemma}
Given $f, h \in \mathcal{L}^1(\mathbb{D})$, we have that 
\begin{equation}
\langle f, h\rangle_H \leq \|f\|_{\mathcal{L}^1(\mathbb{D})} \sup_{g \in \mathbb{D}} \langle g,h\rangle_H.
\end{equation}
\end{lemma}
We shall present the instrumental result and important remark in this paper.
\begin{lemma}
Let $V_n = \pi^{\frac{n}{2}}/\Gamma \left (\frac{n}{2} +1 \right )$ be the volume of the $\ell^2$ unit ball in $\Reals{n}$ and $K$ be the compact subset in $H$. Then for any $v_1,\cdots,v_n \in K$, it holds that   
\begin{equation} 
\left ( \Pi_{k=1}^n \|v_k - P_{k-1}v_k\|_H \right )^{\frac{1}{n}} \leq (n! V_n)^{\frac{1}{n}}\varepsilon_n({\rm co}(K))_H,  
\end{equation} 
where $P_0 = 0$ and $P_k$ is the orthogonal projection onto $H_k = {\rm span} (v_1,\cdots, v_k)$. 
\end{lemma}
\begin{remark}
For the case when $\langle \cdot,\cdot\rangle_H$ is only semi-inner product. We interpret the orthogonal projection as the projection to $\mathcal{N}^\perp$ and apply the usual orthogonal projection. This orthogonal projection will be denoted by $\perp_H$ and used in Theorem \ref{main:thm}.   
\end{remark}
We are in a position to prove our main result: 
\begin{theorem}\label{main:thm} 
The weak OGA applied to the minimization problem: for $G \in \mathcal{G}$, the optimization problem can be achieved by weak OGA and the rate is given as follows: 
\begin{equation}
\inf_{G_n \in \Sigma_{n,M}(\mathbb{D})} \|G - G_n\|_H \lesssim \frac{\|G\|_{\mathcal{L}_1(\mathbb{D})}}{\gamma} \epsilon_n. 
\end{equation}
where $\| \cdot \|_H$ is either semi-norm or norm. 
\end{theorem}
\begin{proof}
The proof closely follows the definite norm case \cite{oga-park2024randomized}. First, let $r_n = G - G_n$ for $n \geq 1$. Since $r_n \perp_{H} H_n$ and $G_n - P_{n-1}G_n \in H_n$, we have that
\begin{equation}
\|r_n\|_H^2 \leq \|r_{n-1} - \alpha (G_n - P_{n-1} G_n) \|_H^2, \quad \forall \alpha \in \Reals{}.
\end{equation}
With $\alpha = \frac{\langle r_{n-1}, G_n - P_{n-1}G_n\rangle_H}{\|G_n - P_{n-1} G_n\|_H^2}$, we get that 
\begin{equation}
\|r_n\|_H^2 \leq \|r_{n-1}\|_H^2 - \frac{\langle r_{n-1}, G_n - P_{n-1}G_n\rangle_H^2}{\|G_n - P_{n-1} G_n\|_H^2}. 
\end{equation}
On the other hand, we have that
\begin{eqnarray*}
\|r_{n-1}\|_H^2 &=& \langle r_{n-1}, G \rangle_H \leq \|G\|_{\mathcal{L}^1(\mathbb{D})} \sup_{G \in \mathbb{D}} |\langle r_{n-1}, G\rangle_H | \\
&\leq& \frac{\|G\|_{\mathcal{L}^1(\mathbb{D})}}{\gamma} |\langle r_{n-1}, G_n \rangle_H | = \frac{\|G\|_{\mathcal{L}^1(\mathbb{D})}}{\gamma} |\langle r_{n-1}, G_n - P_{n-1} G_n \rangle_H |, 
\end{eqnarray*}
where the last equation is due to $r_{n-1} \perp_H H_{n-1}$. Thus, we get that 
\begin{eqnarray}
\frac{\gamma}{\|G\|_{\mathcal{L}^1(\mathbb{D})}} \|r_{n-1}\|_H^2  \leq |\langle r_{n-1}, G_n - P_{n-1} G_n \rangle_H | 
\end{eqnarray}
Thus, we have that
\begin{equation}
\|r_n\|_H^2 \leq \|r_{n-1}\|_H^2 - \frac{\gamma^2\|r_{n-1}\|_H^4}{\|G\|_{\mathcal{L}_1(\mathbb{D})}^2 \|G_n - P_{n-1} G_n\|_H^2} 
\end{equation}
Equivalently, we have that
\begin{equation}
\frac{\gamma^2}{\|G\|_{\mathcal{L}_1(\mathbb{D})}^2} \|r_n\|_H^2 \leq
\frac{\gamma^2}{\|G\|_{\mathcal{L}_1(\mathbb{D})}^2}\|r_{n-1}\|_H^2 \left ( 1  - \frac{\gamma^2\|r_{n-1}\|_H^2}{\|G\|_{\mathcal{L}_1(\mathbb{D})}^2 \|G_n - P_{n-1} G_n\|_H^2} \right ).  
\end{equation}
Thus, by setting 
\begin{equation}
a_n = \frac{\gamma^2}{\|G\|_{\mathcal{L}_1(\mathbb{D})}^2} \|r_n\|_H^2 \quad \mbox{ and } \quad b_n = \|G_n - P_{n-1} G_n\|_H^{-2},  
\end{equation}
we get the recursion relation that 
\begin{equation}
a_n \leq a_{n-1} (1 - b_n a_{n-1}), \quad \forall n \geq 1.
\end{equation}
The argument in \cite{li2023entropy} then gives 
\begin{equation}
a_n \lesssim \epsilon_n^2.  
\end{equation}
Thus, we arrive at the estimate that 
\begin{equation}
\frac{\gamma^2}{\|G\|_{\mathcal{L}_1(\mathbb{D})}^2} \|G - G_n\|_H^2 \lesssim \epsilon_n^2.
\end{equation} 
This completes the proof. 
\end{proof}
Finally, we have the following lemma:
\begin{lemma}\label{num:lem} 
The following estimate holds:
\begin{equation}
\inf_{G_n \in \Sigma_{n,M}(\mathbb{D})} \|G - G_n\|_H \leq \inf_{\psi \in \mathcal{N}} \|G - G_n + \psi\|_{L^2(\Omega \times \Omega)} \leq \inf_{\mu_n \in \Sigma_{n,M}(\mathbb{D})} \|G - \mu_n\|_{L^2(\Omega\times\Omega)}
\end{equation}
\end{lemma}
\begin{proof}
Following the argument given in Lemma \ref{main:lmb}, for any given 
$\mu_n \in \Sigma_{n,M}(\mathbb{D})$ and $\psi \in \mathcal{N}$, we have that 
\begin{equation}
\|G - \mu_n\|_H \leq \|G - \mu_n + \psi\|_{L^2(\Omega\times\Omega)}. 
\end{equation}
Thus, it completes the proof. 
\end{proof}
The upper bound of $\inf_{\mu_n \in \Sigma_{n,M}(\mathbb{D})} \|G - \mu_n\|_{L^2(\Omega \times \Omega)}$ is given by $n^{-\frac{1}{2} - \frac{2k+1}{4d}}$. On the other hand, if the null space $\mathcal{N}$ becomes large, then the estimate for $\|G - G_n\|_H$ can be much smaller by the action of the function $\psi \in \mathcal{N}.$ Thus, as the null space $\mathcal{N}$ gets larger, we expect that the accuracy approximating $G$ in terms of $\|\cdot\|_H$ will improve. This is validated in numerical experiments. Lastly, these estimates can lead to the error in inference as well. We note that the following relation holds between $u$ and $G$: 
\begin{equation}
u = G\star f, 
\end{equation}
from which we can obtain the following result: 
\begin{equation}
\label{eq:approx_rate_kernel}
\inf_{G_n \in \Sigma_{n,M}(\mathbb{D})} \|u - u_n\|_{L^2(\Omega)} = 
\inf_{G_n \in \Sigma_{n,M}(\mathbb{D})} \|G - G_n\|_H    \lesssim \frac{\|G\|_{\mathcal{L}_1(\mathbb{D})}}{\gamma} \epsilon_n. 
\end{equation}
Thus we established that the inference error induced by the Green function approximation by neural network is of the similar estimate. 

\subsection{Shallow neural network in practice}

The approximation theory of shallow neural networks has been well-developed in \cite{nn-barron1993universal,nn-breiman1993hinging,nn-klusowski2018approximation,nn-domingo2021tighter,ogatheory-siegel2024sharp}. However, it is widely observed in practice that shallow neural networks struggle to effectively approximate functions with sharp transitions or rapid oscillations\cite{nn-hong2022activation, nn-chen2022bridging, nn-grossmann2024can}. In \cite{nn-zhang2023shallow}, two important findings were highlighted. First, shallow ReLU neural networks exhibit a low-pass filter nature under finite machine precision, limiting their ability to approximate functions with sharp transitions or rapid oscillations. Second, to accurately approximate discontinuous functions, shallow neural networks require an exponentially increasing width as the dimensionality grows. These two challenges are inevitable when approximating kernel functions using shallow neural networks. Specifically, kernel functions double the dimension of the original problem, and they may exhibit singularities and behave oscillatory, such as Green's function of Helmholtz equations.

\subsection{Point-wise kernel estimation by OGA}

To reduce the difficulties mentioned above, instead of directly approximating the kernel $G(\mathbf{x}, \mathbf{y})$ by a shallow neural network with $[\mathbf{x}, \mathbf{y}] \in \mathbb{R}^{2d}$, we fit several slices of kernel $G(\mathbf{y} | \mathbf{x})$ with only $\mathbf{y} \in \mathbb{R}^{d}$, see Figure \ref{fig:pw}. Specifically, given a set of response location $\{ \mathbf{x}_s \}^{m_u}_{s=1}$, we have the following approximation, 
\begin{equation}
\begin{split}
\label{eq:integral_transform_approx}
u(\mathbf{x}_s) = \int_\Omega G^s(\mathbf{y}) f(\mathbf{y}) \mathrm{d} \mathbf{y},
\end{split}
\end{equation}
where we denote $G(\mathbf{y} | \mathbf{x}_s)$ as $G^s(\mathbf{y})$ for simplicity. The main task in this setting is to train $\{G^s(\mathbf{y})\}_{s=1}^{m_u}$ from data pairs. 
\begin{figure}[h]
\centering
\includegraphics[scale=.7]{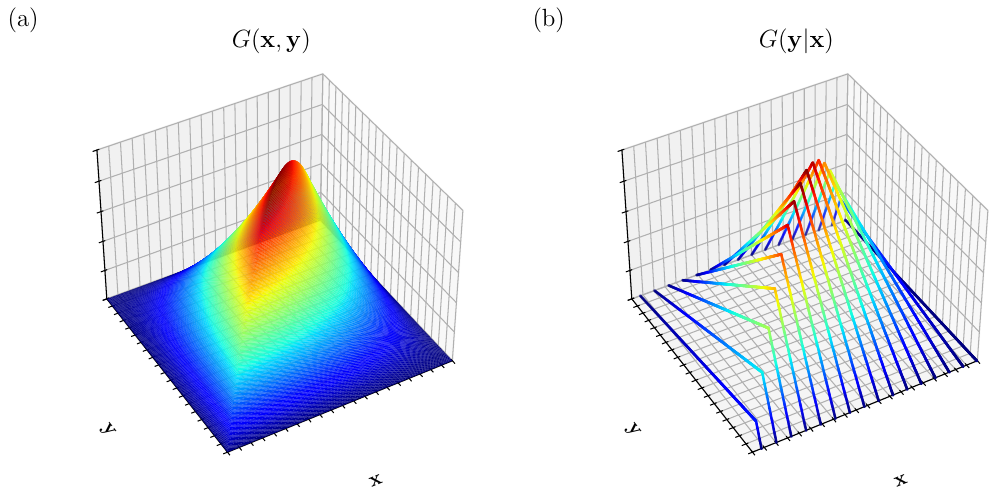}
\caption{(a) Kernel function. (b) Slices of kernel function.} 
\label{fig:pw}
\end{figure}
We name this approach as the point-wise kernel estimation, and the corresponding OGA(PW-OGA) is shown in Algorithm \ref{OGA_rand_kernel_pw}. The advantage of point-wise kernel estimation is mainly on the optimal approximation rate, which is same as the $d$-dimensional function approximation rather than a $2d$-dimensional function approximation.

\begin{algorithm}[h]
\caption{Orthogonal greedy algorithm for point-wise kernel estimation with randomized discrete dictionaries}\label{OGA_rand_kernel_pw}

Given a dataset $\{f_j, u_j\}^{N}_{j=1}$

\begin{algorithmic}
\FOR{$s = 1, 2, \cdots $}
    \STATE{
    Set $G_{s,0} = 0$    
    }
    \FOR{$n = 1, 2, \cdots $}
    \STATE{
    Random dictionary construction 
    \begin{equation}
    \label{eq:rand_kernel_dict}
    \mathbb{D}_{N_R} = \bigl\{ \pm \sigma_{k}(\mathcal{S}(\phi_i) \cdot \mathbf{y} + \beta_i) : \phi_i, \beta_i \sim \mathrm{Uniform}(R), 1 \leq i \leq N_R\}
    \end{equation}}
    \STATE{
    Subproblem optimization
    \begin{equation}
    \label{eq:kernel_max}
    g_n = \argmax_{g \in \mathbb{D}_{N_R}} \langle G_s - G_{s,n-1}, g \rangle_H    
    \end{equation}}
    \STATE{
    Orthogonal projection
    \begin{equation}
    G_{s,n} = P_n(G_s)
    \end{equation}
    }    
    \ENDFOR

\ENDFOR
\end{algorithmic}
\end{algorithm}
More precisely, let 
\begin{equation}
\mathcal{G}_{pw} := \left \{ \tend{G} = \begin{pmatrix}
G^1 \\
\vdots \\
G^{m_u} 
\end{pmatrix} \in \left [ \mathcal{L}_1(\mathbb{D}) \right ]^{m_u} : \exists \tenq{G} \in \mathbb{R}^{m_u \times m_f} \mbox{ such that } \tenq{G} \tend{f} = \begin{pmatrix} 
(G^1 \star f)(\mathbf{x_1})   \\ 
\vdots \\
(G^{m_u} \star f)(\mathbf{x_{m_u}}) 
\end{pmatrix} \right \}. 
\end{equation}
In this setting, for the $n$-neurons approximation $G^{s}_n$, we see that 
\begin{equation}\label{min:pwoga} 
\min_{\tend{G_n} \in \left[ \mathcal{L}_1(\mathbb{D}) \right ]^{m_u}} \|\tend{G} - \tend{G_n}\|_H = 
\min_{G^1_n,\cdots G^{m_u}_n} \sum_{s=1}^{m_u} \|G^s - G_n^s\|_{H_s},
\end{equation}
where 
\begin{equation}
\|G^s - G^s_n\|_{H_s} = \frac{1}{N} \sum_{j=1}^N \|(G\star f_j)(\mathbf{x}_s) - G^s_n \star f_j\|_{L^2(\Omega)}.
\end{equation}
Following the argument in Theorem \ref{main:thm}, we can establish that 
\begin{equation}
\label{eq:Gs_H_norm_bound}
\|G^s - G_n^s\|_{H_s} \leq \inf_{\phi \in \mathcal{N}_s} \|G^s - G_n^s + \phi\|_{L^2(\Omega)} \leq \|G^s - G_n^s\|_{L^2(\Omega)} \lesssim n^{-\frac{1}{2} - \frac{2k+1}{2d}},
\end{equation} 
where 
\begin{equation}
\label{eq:null_space2}
\mathcal{N}_s = \{ G \in \mathcal{G} : (G \star f) (\mathbf{x}_s) = 0,~ \forall f \in \mathcal{K} \}.
\end{equation}
Thus, we establish the approximation error in the point-wise kernel estimation by OGA in the following theorem. 
\begin{theorem}
The weak OGA for the optimization problem \eqref{min:pwoga} has the following approximation error estimate:
\begin{equation}
\inf_{\tend{G_n} \in \left[ \mathcal{L}_1(\mathbb{D}) \right ]^{m_u}} \|G - \tend{G_n}\|_H \lesssim c(m_u) n^{-\frac{1}{2} - \frac{2k+1}{2d}}.
\end{equation}
where $c(m_u)$ is a constant that is dependent on $m_u$ and $\| \cdot \|_H$ is either semi-norm or norm. 
\end{theorem}
\begin{remark}
Generally, the upper bound shows that if $m_u$ is smaller than $n^{\frac{2k+1}{4d}}$, then it is advantageous to use {\rm PW-OGA}. On the other hand, the effect of the null space in actual error significantly helps, making the use of {\rm PW-OGA} favorable.  
\end{remark}

\section{Numerical Results}
\label{sec:result}

In this section, we present the numerical results of the OGA for kernel estimation and point-wise kernel estimation across various linear operator learning problems to demonstrate its effectiveness. The test problems encompass both irregular domain cases and high-dimensional examples, including those with and without exact kernels. In all cases, the OGA is performed with the shallow ReLU neural network, and the random dictionary size is set to 512 at each iteration. For comparison, we use GreenLearning (GL) \cite{green-boulle2022gl}, Fourier neural operator (FNO) \cite{no-li2020fourier}, and Deep operator network (DON) \cite{don-lu2021learning} as baseline methods. The details of these baseline methods are provided in \ref{app:baseline}. We implement all algorithms in Python with Pytorch package. All computations are performed on a single Nvidia RTX A6000 GPU with 48GB memory.

To assess the accuracy of different methods, we compute the relative $L_2$ error of the solutions in the test sets and the relative $L_2$ error of the kernels, when exact kernel functions are available. The relative $L_2$ error of solutions $\epsilon_u$ and the relative $L_2$ error of kernel $\epsilon_G$ are defined as follows,
\begin{equation} 
\label{eq:rl2_metric}
\epsilon_{u} = \frac{1}{N} \sum_{j=1}^{N} \frac{\| u_j - \Tilde{u}_j \|_{L^2(\Omega)}}{ \|u_j\|_{L^2(\Omega)} }, \quad \epsilon_{G} = \frac{\| G - \Tilde{G} \|_{L^2(\Omega \times \Omega)}}{ \| G \|_{L^2(\Omega \times \Omega)}},
\end{equation}
where $\Tilde{u}$ and $\Tilde{G}$ denotes the corresponding approximants by neural network. To visualize the results from different methods, we also compute the point-wise absolute error of a sample solution or a kernel, denoted as $E_u(\mathbf{x})$ and $E_G(\mathbf{x}, \mathbf{y})$ respectively, 
\begin{equation} 
\label{eq:abs}
E_{u}(\mathbf{x}) = | u(\mathbf{x}) - \Tilde{u}(\mathbf{x}) |, \quad E_{G}(\mathbf{x}, \mathbf{y}) = | G(\mathbf{x}, \mathbf{y}) - \Tilde{G}(\mathbf{x}, \mathbf{y}) |.
\end{equation}

Numerical integration is necessary to evaluate the inner product of the function, the intergal, orthogonal projection, and the relative error $L_2$. We employ the Monte-Carlo integration method\cite{mc-binder2012monte} in all cases.

\subsection{Numerical results for one-dimensional problems}

We select the one-dimensional Poisson equation and Helmholtz equation as test examples. The forcing functions are sampled from a Gaussian process\cite{gp-williams2006gaussian} with zero mean and a length scale of 0.01 by Chebfun software\cite{cheb-Driscoll2014}. The discrete forcing vectors are evaluated on a uniform grid that spans $[-1,1]$ with $501$ points. The exact kernels are available for both problems, and the corresponding solutions are computed using numerical kernel integration. We split the data into training and testing sets of 500 and 200 samples, respectively.

\paragraph{One-dimensional Poisson equation}
We first give the example of a one-dimensional Poisson equation with homogeneous boundary conditions:
\begin{equation} 
\label{eq:poisson1D}
\begin{split}
-\frac{d^2 u}{d x^2} & = f(x), \quad x \in [0,1] \\
u(0) & = u(1) = 0.
\end{split}
\end{equation}
The associated Green's function is available in closed-form\cite{no-li2020gno} as 
\begin{equation}
    \begin{split} \label{eq:poisson_green}
        G(x,y) = 
        \left\{
            \begin {aligned}
                & x(y-1) \quad & x \leq y \\
                & y(x-1) \quad & x > y                
            \end{aligned}
        \right.
    \end{split}
\end{equation}
where $(x, y) \in [0,1]^2$. 

\begin{figure}[h!]
    \centering
    \includegraphics[scale=0.6]{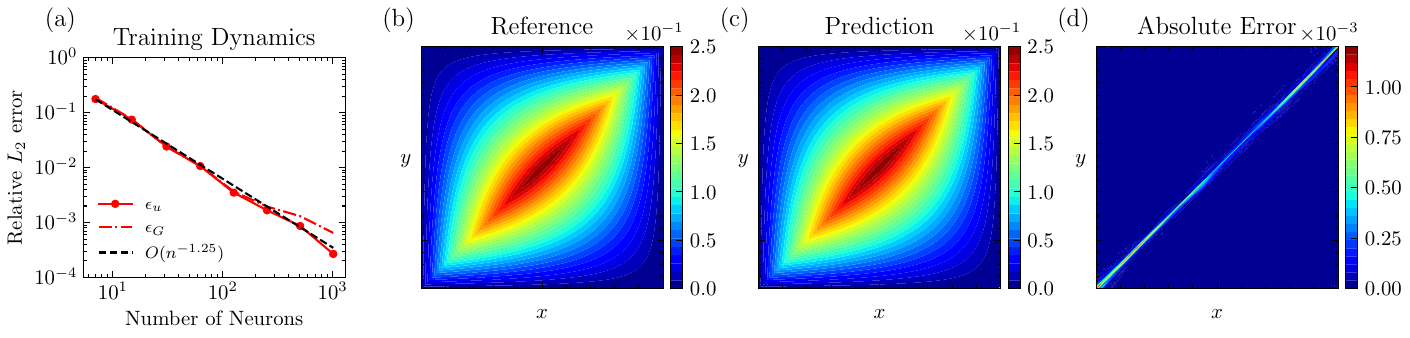}
    \caption{Result of 1D Poisson problem: (a) The training dynamics of OGA method. (b) Exact Green's function. (c) Learned Green's function. (d) Absolute error $E_G$.}
    \label{fig:poisson1D_result}
\end{figure}


\begin{table}[h!]
\centering
\begin{tabular}{ccc}
\toprule 
    Method & $\epsilon_G$  & $\epsilon_u$ \\
\midrule
OGA & \textbf{5.7327e-04} & \textbf{1.2411e-04} \\ 
GL & 5.8218e-03 & 5.6763e-03 \\ 
FNO & - & 3.5094e-02 \\ 
DON & - & 1.8072e-01 \\
\bottomrule
\end{tabular}
\caption{The numerical error comparison for different methods on one-dimensional Poisson problem.}
\label{table:poisson1D_compare}
\end{table}

\paragraph{One-dimensional Helmholtz equation} We consider the one-dimensional Helmholtz equation with homogeneous boundary conditions:
\begin{equation} 
\label{eq:helmholtz1D}
\begin{split}
\frac{d^2 u}{d x^2} + K^2 u & = f(x), \quad x \in [0,1] \\
u(0) & = u(1) = 0.
\end{split}
\end{equation}
The corresponding Green's function has the closed-form\cite{green-boulle2022gl} as
\begin{equation}
    \begin{split} \label{eq:helomholz_green}
        G(x,y) = 
        \left\{
            \begin {aligned}
                & \frac{\sin(K x) \sin(K (y-1))}{K\sin(K)}  \quad & x \geq y \\
                & \frac{\sin(K y) \sin(K (x-1))}{K\sin(K)} \quad & x < y                
            \end{aligned}
        \right.
    \end{split}
\end{equation}
where $(x, y) \in [0,1]^2, K=15$. 

\begin{figure}[h!]
    \centering
    \includegraphics[scale=0.6]{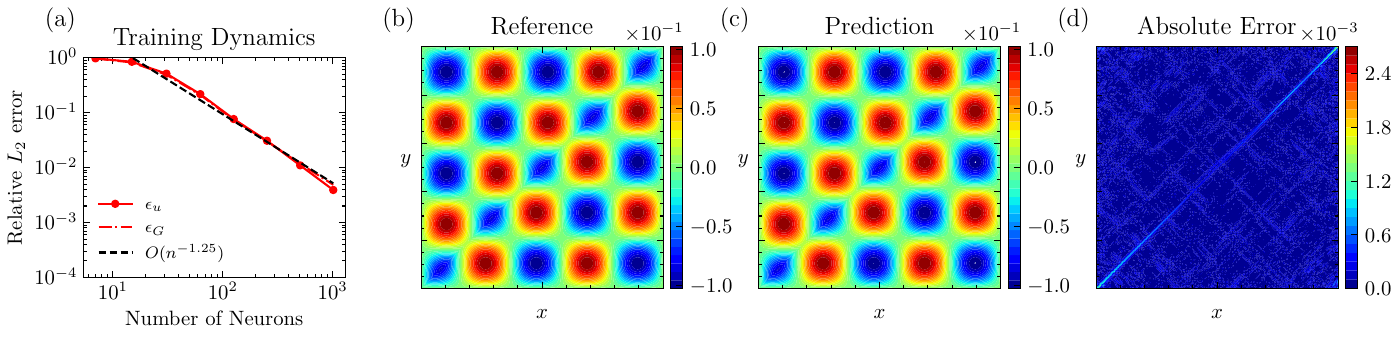}
    \caption{Result of 1D Helmholtz problem: (a) The training dynamics of OGA method. (b) Exact Green's function. (c) Learned Green's function. (d) Absolute error $E_G$.}
    \label{fig:helmhotz1D_result}
\end{figure}


\begin{table}[h!]
\centering
\begin{tabular}{ccc}
\toprule 
    Method & $\epsilon_G$  & $\epsilon_u$ \\
\midrule
OGA & \textbf{2.3186e-03} & \textbf{1.4509e-03} \\ 
GL & 1.9849e-02 & 2.0166e-02 \\ 
FNO & - & 3.8912e-02 \\ 
DON & - & 2.4209e-01 \\ 
\bottomrule
\end{tabular}
\caption{The numerical error comparison for different methods on one-dimensional Helmholtz problem.}
\label{table:helmholtz1D_compare}
\end{table}

The numerical results for the one-dimensional Poisson problem are presented in Figure \ref{fig:poisson1D_result} and Table \ref{table:poisson1D_compare}. Figure \ref{fig:poisson1D_result}(a) illustrates the relationship between the number of neurons and the approximation errors $\epsilon_G$ and $\epsilon_u$. The optimal approximation rate (\ref{eq:approx_rate_kernel}) is shown as a black dotted reference line. Both learned kernel and solution functions achieve the optimal approximation rate. Compared to baseline methods, OGA demonstrates an order of magnitude improvement on relative $L_2$ error. Figure \ref{fig:helmhotz1D_result} and Table \ref{table:helmholtz1D_compare} summarize the numerical results for the one-dimensional Helmholtz problem. The oscillatory nature of the Green's function of Helmholtz equation, as shown in Figure \ref{fig:helmhotz1D_result} (b), poses a greater challenge for a shallow neural networks compared to the counterpart of Poisson equation. However, OGA still exhibits the optimal convergence rate and significant improvements in accuracy.

\subsection{Numerical results for two-dimensional problems}
\label{sec:oga-2D}

For two-dimensional problems, we also choose Cosine, Poisson, and Helmholtz problems as test examples. To verify the effectiveness of OGA for kernel estimation on different geometry, we create two meshes on a unit disk domain(833 nodes) and an irregular H-shaped domain(997 nodes) by Gmsh software\cite{gmsh-geuzaine2009gmsh}, see Figure \ref{fig:mesh2D}. We generate random forcing functions using \texttt{randnfundisk} for the disk domain and \texttt{randnfun2} for the H-shaped domain from Chebfun software \cite{cheb-Driscoll2014}, where the wavelength parameter is set to $0.2$ for all problems. The discrete forcing functions are evaluated at each node. We directly calculate the response functions of the cosine kernel by numerical integration. For Poisson and Helmholtz problems, the corresponding solutions are solved by finite element method using DOLFINx software\cite{fenics-dolfinx2023preprint}. We split the data into training and testing sets of 1000 and 500 samples. 

\begin{figure}[h!]
    \centering
    \includegraphics[scale=0.4]{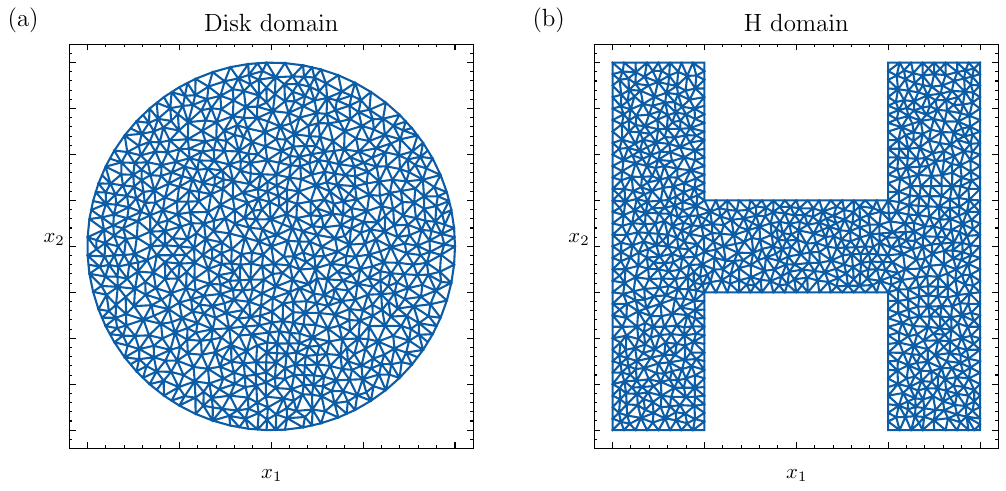}
    \caption{Computational domain of 2D problems: (a) unit disk domain. (b) H shaped domain.}
    \label{fig:mesh2D}
\end{figure}

\paragraph{Two-dimensional Cosine kernel integral} Two-dimensional Cosine kernel integral on domain $\Omega$ with wave number $k$:
\begin{equation} 
\label{eq:cos2D}
\begin{split}
u(\mathbf{x}) &= \int_{\Omega} \cos(k\pi \|\mathbf{x} - \mathbf{y} \|_2) f(\mathbf{y}) \mathrm{d} \mathbf{y}, \quad \mathbf{x} \in \Omega.
\end{split}
\end{equation}

\paragraph{Two-dimensional Poisson equation} Two-dimensional Poisson equation defined on domain $\Omega$ with homogeneous boundary conditions:
\begin{equation} 
\label{eq:poisson2D}
\begin{split}
-\nabla \cdot (\nabla u(\mathbf{x})) &= f(\mathbf{x}), \quad \mathbf{x} \in \Omega, \\
u(\mathbf{x}) &= 0, \quad \mathbf{x} \in \partial \Omega.
\end{split}
\end{equation}

\paragraph{Two-dimensional Helmholtz equation} Two-dimensional Helmholtz equation defined on domain $\Omega$ with homogeneous boundary conditions:
\begin{equation} 
\label{eq:helmholtz2D}
\begin{split}
\nabla \cdot (\nabla u(\mathbf{x})) + K^2 u(\mathbf{x}) &= f(\mathbf{x}), \quad \mathbf{x} \in \Omega, \\
u(\mathbf{x}) &= 0, \quad \mathbf{x} \in \partial \Omega,
\end{split}
\end{equation}
where $K=10$.

\subsubsection{Results on OGA for kernel estimation}

The numerical results for the cosine kernel with wave number $\{1, 2, 4\}$ on the unit disk domain are plotted in Figure \ref{fig:cos2D_converge}. We observe that numerical approximations can achieve the optimal convergence order in both cases. However, more neurons are needed to go through the "warm-up" stage for a high wave number cosine kernel, as the kernel becomes more oscillatory.

\begin{figure}[h!]
    \centering
    \includegraphics[scale=0.5]{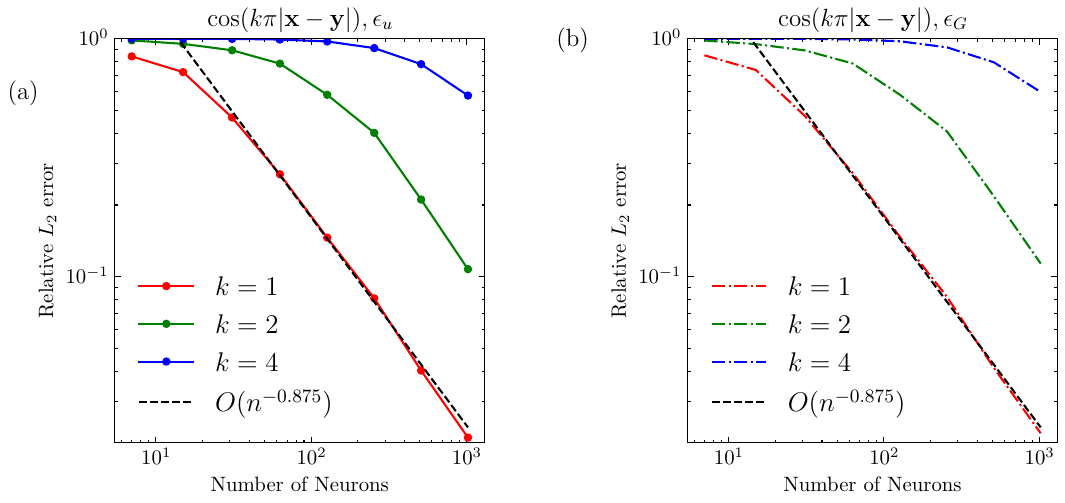}
    \caption{Training dynamics of cosine kernel integral: (a) relative $L_2$ error of solutions. (b) relative $L_2$ error of kernel}
    \label{fig:cos2D_converge}
\end{figure}

Figure \ref{fig:ph2D_converge} displays the numerical result of the 2D Poisson and Helmholtz problems. The corresponding kernels have singularities on $\mathbf{x} = \mathbf{y}$ \cite{green-lin2023bi}. Although we will learn the bounded approximations numerically, the approximations still have a very sharp transition in each case, which requires even more neurons to achieve the optimal convergence rate and leads to a long "warm-up" stage.

\begin{figure}[h!]
    \centering
    \includegraphics[scale=0.5]{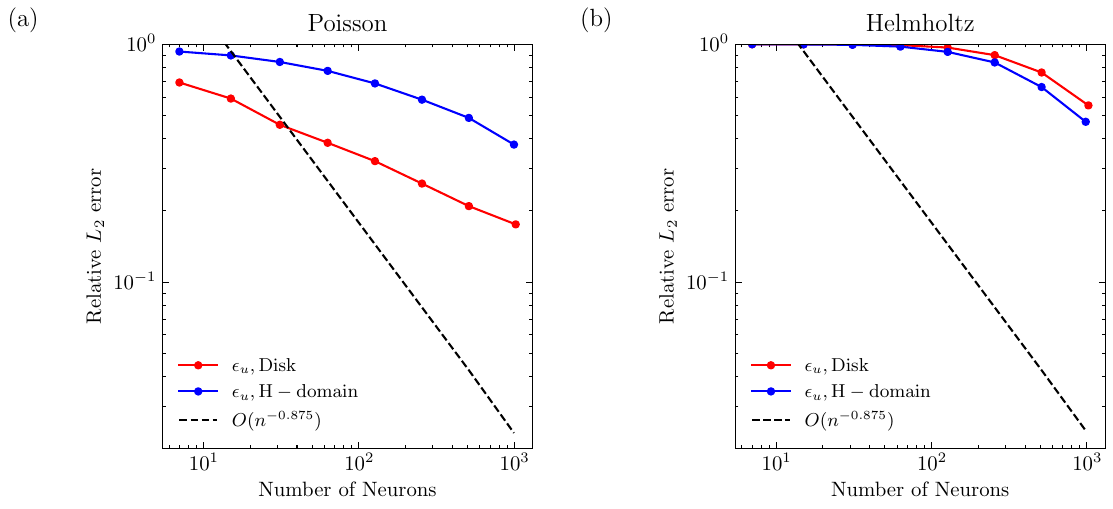}
    \caption{Training dynamics of 2D Poisson and Helmholtz problems: (a) Poisson problems. (b) Helmholtz problems}
    \label{fig:ph2D_converge}
\end{figure}

The numerical results show that more neurons are needed for the shallow ReLU neural network to approximate functions with sharp transitions or rapid oscillations, which is consistent with \cite{nn-zhang2023shallow}. In practice, under the limitation of finite machine precision, we may fail to solve linear systems at OGA projection step with a large number of neurons. These two difficulties prevent us from learning better results with OGA on directly learning kernels from data.

\subsubsection{PW-OGA results for two-dimensional problems}
\label{sec:PW-OGA_2D}
We test the same two-dimensional problems for PW-OGA. The training dynamics of PW-OGA is showed in Figure \ref{fig:cos2D_converge_pw} and Figure \ref{fig:ph2D_converge_pw}. We use shallow ReLU neural network to approximate $G(\mathbf{y} | \mathbf{x})$ which is a $d$-dimensional function as forcing functions. The optimal convergence rate of 2D function fitting is $O(n^{-1.25})$. For the unit disk cosine problem, similar to the direct OGA method, a high wave number cosine kernel still needs to go through a relatively long "warm-up" stage. After the warm-up stage, both $\epsilon_u$ and $\epsilon_G$ could achieve optimal approximation rate. However, we observe the "overfitting" phenomenon as the number of neurons increases, where $\epsilon_u$ converges even faster while $\epsilon_G$ diverges, see Figure \ref{fig:cos2D_converge_pw}. The overfitting phenomenon in this case is not attributable to overfitting the train dataset $\{f_j\}^{N}_{j=1}$, but rather to the function space $\mathcal{K}$(\ref{eq:data_space}). In terms of operator learning, overfitting to a function space is acceptable for applications where no out-of-range data exists. Since the exact form of kernel function is not available, only the accuracy of the response functions matters. 

For the Poisson and Helmholtz problems, we obtain the similar result as cosine problems, see Figure \ref{fig:ph2D_converge_pw}, and compare it with OGA for kernel estimation, GL and DeepONet. The accuracy of PW-OGA achieves the order magnitude of improvement in terms of operator learning compared to others; see Table \ref{table:result2D}. The sample response functions are shown in Figure \ref{fig:poisson2D} \ref{fig:poisson2DH} \ref{fig:helmholtz2D} \ref{fig:helmholtz2DH}.

\begin{figure}[h!]
    \centering
    \includegraphics[scale=0.5]{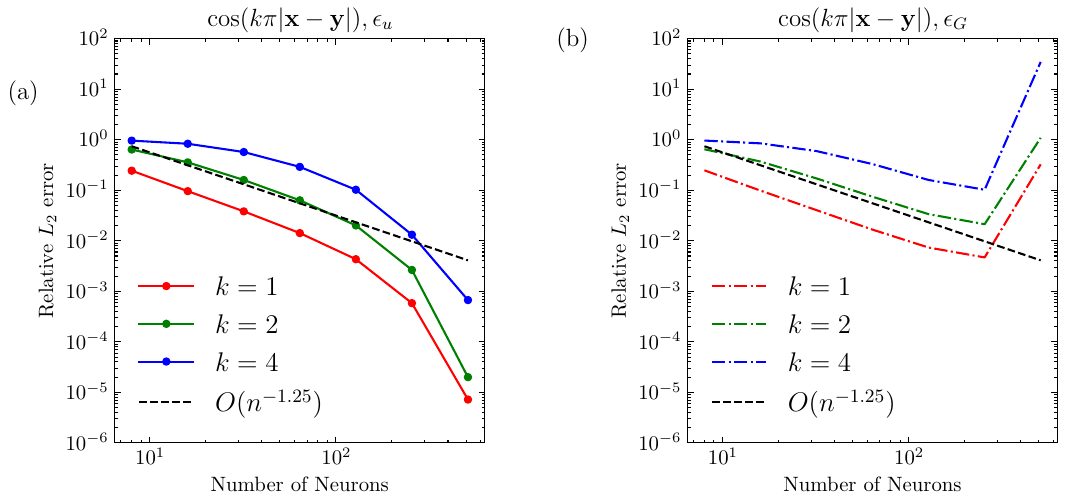}
    \caption{Training dynamics of 2D test problems with pair-wise kernel estimation. (a) relative $L_2$ error of solutions. (b) relative $L_2$ error of kernel}
    \label{fig:cos2D_converge_pw}
\end{figure}

\begin{figure}[h!]
    \centering
    \includegraphics[scale=0.5]{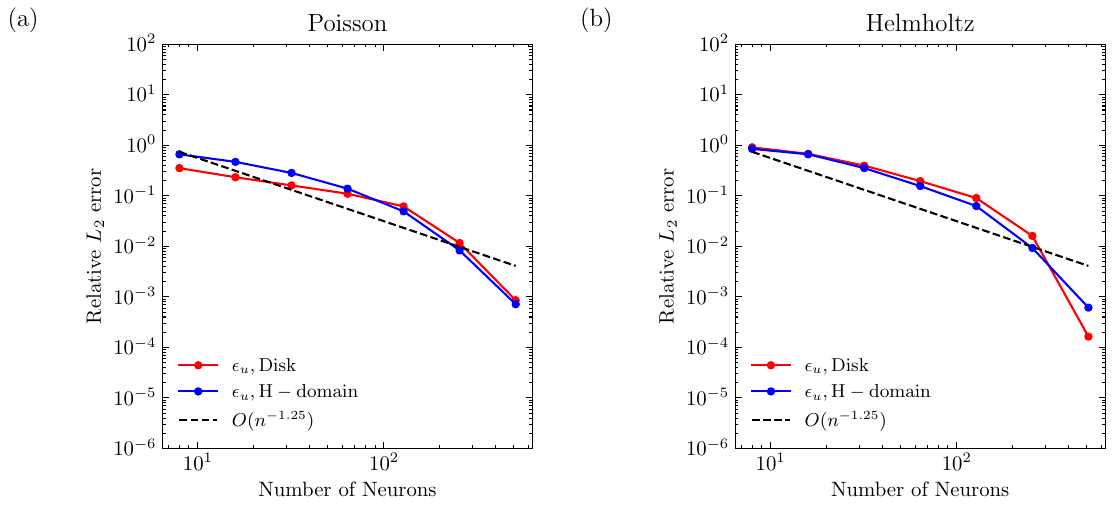}
    \caption{Training dynamics of 2D test problems with pair-wise kernel estimation. (a) Poisson problems. (b) Helmholtz problems}
    \label{fig:ph2D_converge_pw}
\end{figure}

\begin{table}[h!]
\centering
\begin{tabular}{ lcccc }
    \toprule
         & \multicolumn{2}{c}{Poisson}    & \multicolumn{2}{c}{Helmholtz}\\
    \cmidrule(lr){2-3}\cmidrule(lr){4-5}
         &  Disk & H-domain & Disk & H-domain \\
    \midrule
    PW-OGA & \textbf{8.5628e-04} & \textbf{7.1716e-04} & \textbf{1.6378e-04} & \textbf{6.1449e-04} \\ 
    OGA & 1.7499e-01 & 3.7873e-01 & 5.5332e-01 & 4.6881e-01 \\ 
    GL & 4.0472e-02 & 7.1289e-02 & 1.4822e-01 & 2.0567e-01 \\ 
    DON & 3.3877e-01 & 4.9925e-01 & 2.1499e-01 & 2.4925e-01 \\
    \bottomrule
\end{tabular}
\caption{The numerical error for 2D problems on $\epsilon_u$, where PW-OGA represents OGA for point-wise kernel estimation.}
\label{table:result2D}
\end{table}

\begin{figure}[h!]
    \centering
    \includegraphics[scale=0.7]{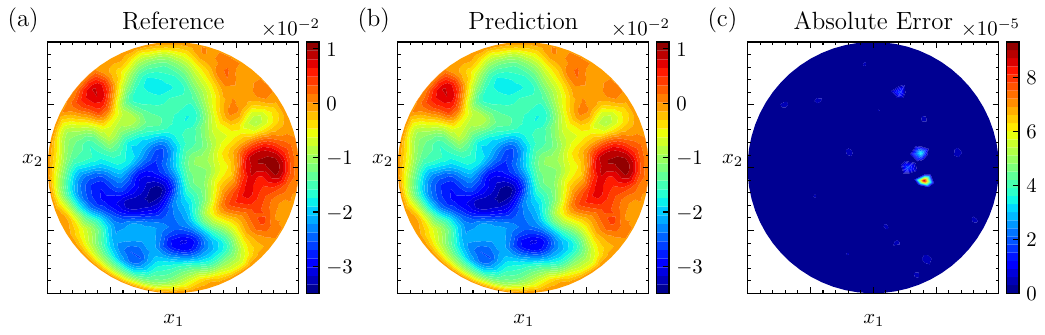}
    \caption{Sample result of 2D poisson problem on unit disk: (a) Reference FEM result $u$. (b) PW-OGA learned result $\Tilde{u}$. (c) Absolute error $|u - \Tilde{u}|$.}
    \label{fig:poisson2D}
\end{figure}

\begin{figure}[h!]
    \centering
    \includegraphics[scale=0.7]{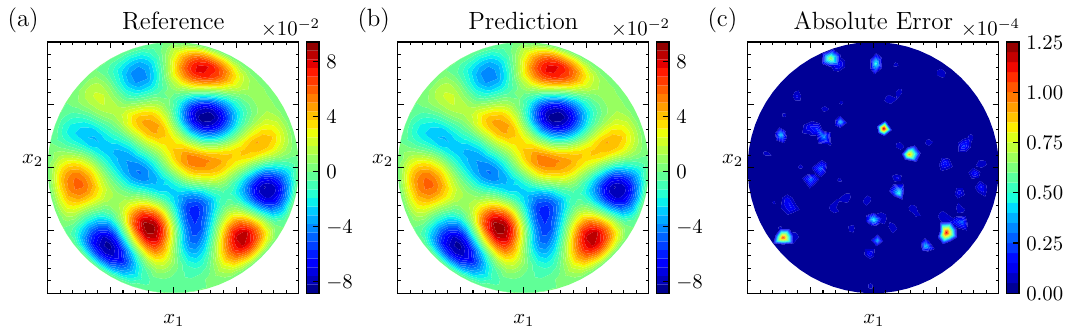}
    \caption{Sample result of 2D helmholtz problem on unit disk: (a) Reference FEM result $u$. (b) PW-OGA learned result $\Tilde{u}$. (c) Absolute error $|u - \Tilde{u}|$.}
    \label{fig:helmholtz2D}
\end{figure}

\begin{figure}[h!]
    \centering
    \includegraphics[scale=0.7]{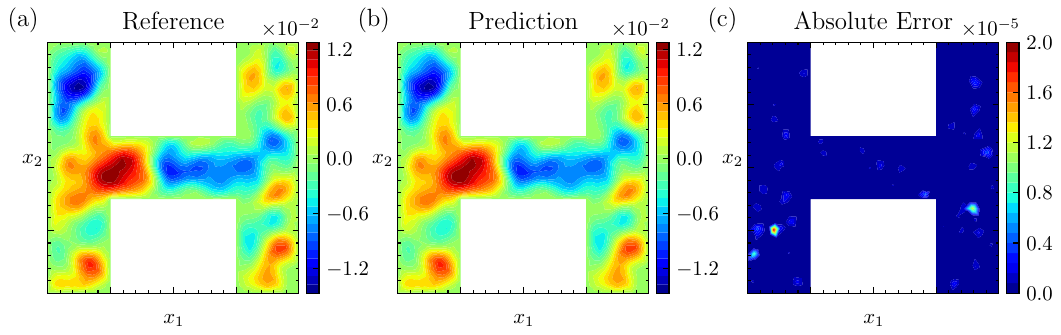}
    \caption{Sample result of 2D poisson problem on H shaped domain: (a) Reference FEM result $u$. (b) PW-OGA learned result $\Tilde{u}$. (c) Absolute error $|u - \Tilde{u}|$.}
    \label{fig:poisson2DH}
\end{figure}

\begin{figure}[h!]
    \centering
    \includegraphics[scale=0.7]{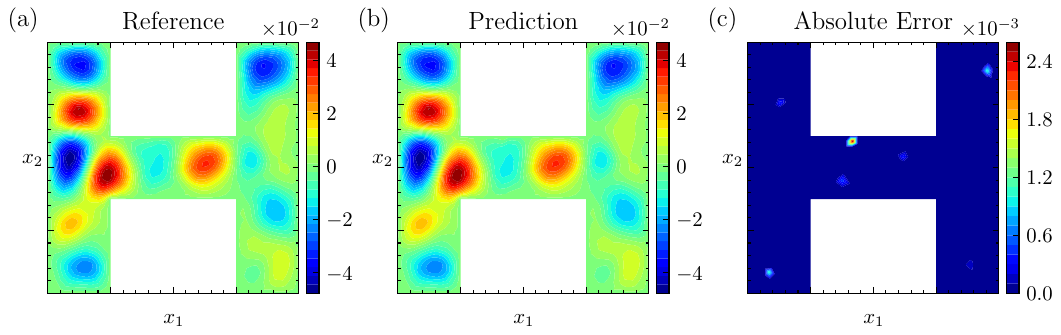}
    \caption{Sample result of 2D helmholtz problem on H shaped domain: (a) Reference FEM result $u$. (b) PW-OGA learned result $\Tilde{u}$. (c) Absolute error $|u - \Tilde{u}|$.}
    \label{fig:helmholtz2DH}
\end{figure}

\subsection{PW-OGA results for three-dimensional problems}
To further verify the effectiveness of OGA for point-wise kernel estimation in high dimension, we tested two 3D problems defined on the unit cube. The computational domain is discretized by \texttt{UnitCubeMesh} function \cite{gmsh-geuzaine2009gmsh} with 17 nodes on each axis. The 3D random forcing functions are generated by \texttt{GaussianRandomField} function from the Julia package\cite{grf-robbe2023gaussianrandomfields} with correlation length $0.2$, and the corresponding response functions are calculated from the Monte-Carlo integration method. We split the data into training and testing sets, each containing 1000 samples.

\paragraph{Three-dimensional cosine kernel} The first example is the 3D cosine kernel integral, where cosine kernel is a smooth kernel
\begin{equation} 
\label{eq:cos3D}
\begin{split}
G(\mathbf{x}, \mathbf{y}) &= \cos(2\pi \|\mathbf{x} - \mathbf{y}\|_2), \quad \mathbf{x}, \mathbf{y} \in [0,1]^3. \\
\end{split}
\end{equation}

\paragraph{Three-dimensional oscillatory singular kernel} The second example is a 3D oscillatory singular kernel integral,
\begin{equation} 
\label{eq:logcos3D}
\begin{split}
G(\mathbf{x}, \mathbf{y}) &= \log(\|\mathbf{x} - \mathbf{y}\|_2) \cos(2\pi \|\mathbf{x} - \mathbf{y}\|_2), \quad \mathbf{x}, \mathbf{y} \in [0,1]^3. \\
\end{split}
\end{equation}

The numerical results for 3D problems are shown in Figure (\ref{fig:ogapw-converge3D}), and the sample learned responses are shown in Figure (\ref{fig:cos3D}) (\ref{fig:logcos3D}). The optimal approximation rate of the shallow ReLU neural netowrk is $O(n^{-1})$ for the 3D function approximation. Table \ref{table:result3D} shows PW-OGA surpasses both FNO and DON significantly in both cases. As we discussed in \ref{sec:PW-OGA_2D}, $\epsilon_G$ roughly follows the optimal rate and gradually diverges with more neurons. 

\begin{figure}[h!]
    \centering
    \includegraphics[scale=0.7]{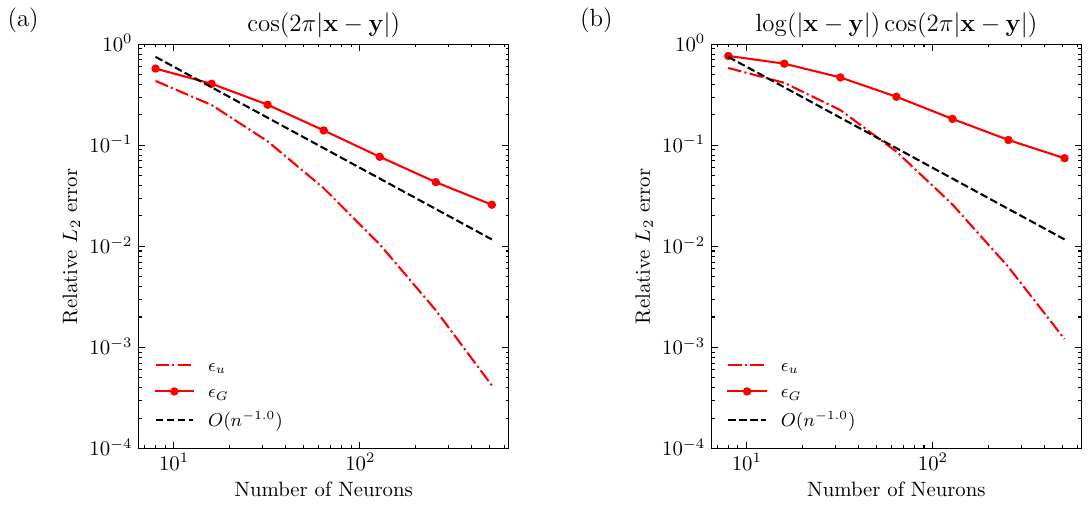}
    \caption{Training dynamics($\epsilon_u$) of 3D test problems with pair-wise kernel estimation.}
    \label{fig:ogapw-converge3D}
\end{figure}

\begin{figure}[h!]
    \centering
    \includegraphics[scale=0.7]{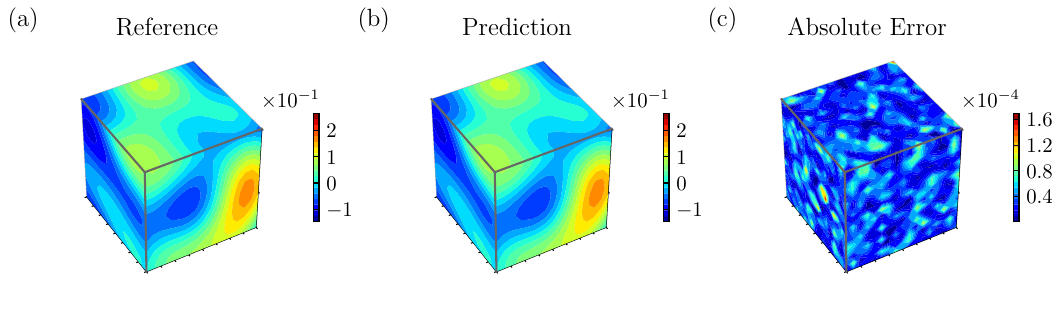}
    \caption{Sample responses of 3D smooth kernel: (a) Reference response $u$. (b) PW-OGA learned response $\Tilde{u}$. (c) Absolute error $|u - \Tilde{u}|$.}
    \label{fig:cos3D}
\end{figure}

\begin{figure}[h!]
    \centering
    \includegraphics[scale=0.7]{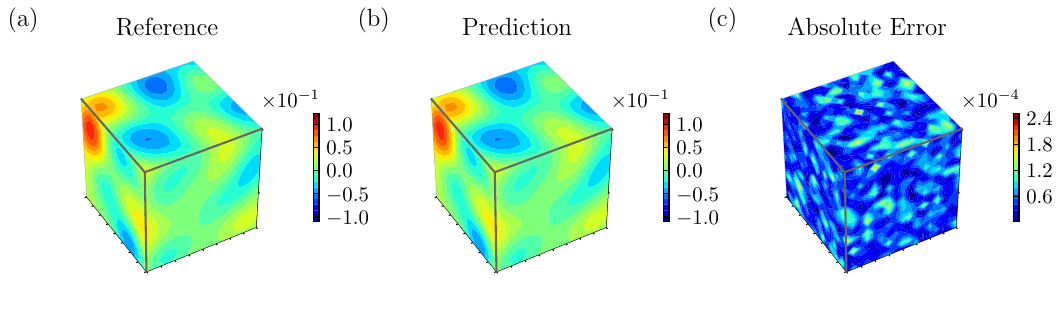}
    \caption{Sample responses of 3D oscillatory singular kernel: (a) Reference response $u$. (b) PW-OGA learned response $\Tilde{u}$. (c) Absolute error $|u - \Tilde{u}|$.}
    \label{fig:logcos3D}
\end{figure}

\begin{table}[h!]
\centering
\begin{tabular}{ lcccc }
    \toprule
         & \multicolumn{2}{c}{Smooth}    & \multicolumn{2}{c}{Oscillatory singular}\\
    \cmidrule(lr){2-3}\cmidrule(lr){4-5}
         &  $\epsilon_u$ & $\epsilon_G$ & $\epsilon_u$ & $\epsilon_G$ \\
    \midrule
    OGApw & \textbf{4.2071e-04} &  2.5830e-02 & \textbf{1.2106e-03} & 7.4495e-02 \\ 
    FNO & 4.3766e-02 & - & 3.6210e-02 & - \\ 
    DON & 3.9122e-02 & - & 5.8760e-02 & - \\ 
    \bottomrule
\end{tabular}
\caption{The numerical error point-wise OGA in average relative $L_2$ norms of Kernel function and responses.}
\label{table:result3D}
\end{table}

\subsubsection{Overfitting treatment for PW-OGA} \label{sec:overfit}

Theoretically, both $\epsilon_u$ and $\epsilon_G$ should converge under optimal approximation rate (\ref{eq:woga_approx_rate}) with a sufficiently large and diverse dataset, i.e., when $\mathcal{N} = \{0\}$. However, for $\mathcal{N} \neq \{0\}$, overfitting is expected, as dicussed in Section \ref{sec:PW-OGA_2D}, based on theoretical analysis (\ref{eq:G_H_norm_bound}) and (\ref{eq:Gs_H_norm_bound}). For a given reference kernel $G$, numerical results show that the neural network approximation $\Tilde{G}$ yields a small inference error, i.e., small $\epsilon_u$. However, we observe that $\epsilon_G$ may not be small when $\mathcal{N} \neq \{0\}$, which is consistent with theory, presented in Lemma \ref{num:lem}. Therefore, a larger null space $\mathcal{N}$ will lead to a faster convergence on $\epsilon_u$, but it also tends to cause $\epsilon_G$ to diverge.

To verify this overfitting phenomenon, we generate random function $\{f_j\}^N_{j=1}$ as training data on the unit disk domain as described in (\ref{sec:oga-2D}) with different wavelength parameter $\lambda = {0.1, 0.2, 0.5}$. Larger wavelengths imply more correlated data, resulting in a larger $\mathcal{N}$. Specifically, we construct a matrix $F$, where each column represents the vector $\tend{f_j}$, the evaluation of $f_j$ at the nodes on the unit disk. To measure the linear dependence among the columns, we perform Singular Value Decomposition(SVD) on $F$. A significant number of near-zero singular values indicates that $F$ is low-rank, and the columns are highly dependent. As a test example, we use the cosine kernel integral $u(\mathbf{x}) = \int_{\Omega} \cos(4\pi \|\mathbf{x} - \mathbf{y} \|_2) f(\mathbf{y}) \mathrm{d} \mathbf{y}$. Figure (\ref{fig:overfit2D}) verifies that the overfitting phenomenon correlates strongly with the column dependency of $F$. For $\gamma = 0.5$, the matrix $F$ has many near-zero singular values, whereas for $\gamma=0.1$, all singular values of $F$ are non-zero. Consequently, $\epsilon_u$ converges faster as $\gamma$ increases, while $\epsilon_G$ tends to diverge. Using more linearly independent training data can help mitigate overfitting in PW-OGA, though it may deteriorate $\epsilon_u$. Nevertheless, PW-OGA still outperforms all other baseline methods, as shown in Tables \ref{table:overfit2D} \ref{table:overfit3D}.

\begin{figure}[h!]
    \centering
    \includegraphics[scale=0.7]{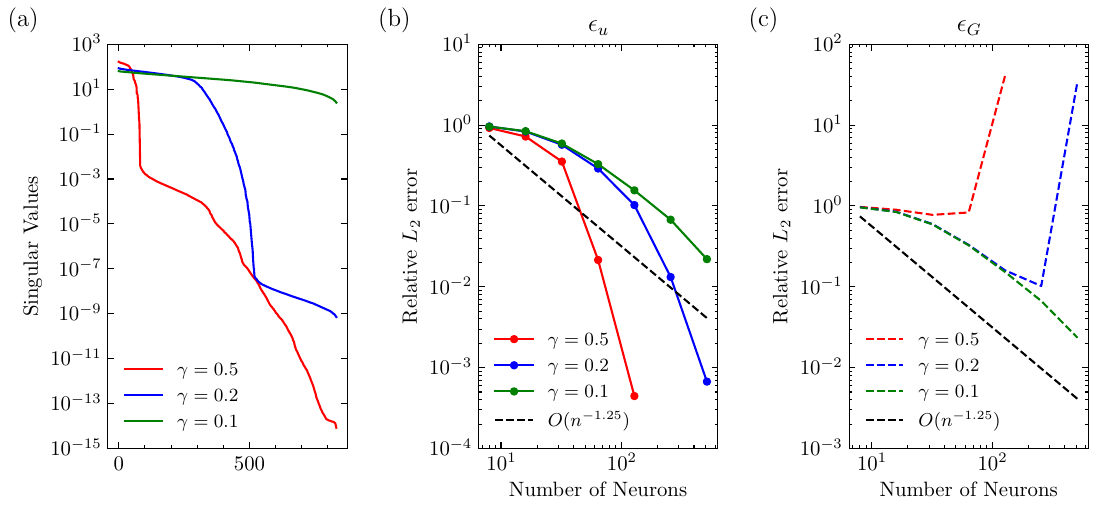}
    \caption{Overfitting phenomenon analysis: (a) Singular values of $F$, arranged in decreasing order. (b) Training dynamics of $\epsilon_u$ (c) Training dynamics of $\epsilon_G$}
    \label{fig:overfit2D}
\end{figure}

\begin{table}[h!]
\centering
\begin{tabular}{ lcccc }
    \toprule
         & \multicolumn{2}{c}{$\gamma = 0.2$}    & \multicolumn{2}{c}{$\gamma = 0.1$}\\
    \cmidrule(lr){2-3}\cmidrule(lr){4-5}
     Poisson   &  Disk & H-domain & Disk & H-domain \\
    \midrule
    PW-OGA & \textbf{8.5628e-04} & \textbf{7.1716e-04} & \textbf{9.1918e-02} & \textbf{6.8071e-02} \\ 
    GL & 4.0472e-02 & 7.1289e-02 &  1.9927e-01 & 1.5573e-01 \\ 
    DON & 3.3877e-01 & 4.9925e-01 & 7.3519e-01 & 8.9244e-01 \\
    \midrule
     Helmholtz &  Disk & H-domain & Disk & H-domain  \\
    \midrule
    PW-OGA & \textbf{1.6378e-04} & \textbf{6.1449e-04} & \textbf{1.1407e-01} & \textbf{1.1402e-01} \\ 
    GL & 1.4822e-01 & 2.0567e-01 &  2.4238e-01 & 2.6776e-01 \\ 
    DON & 2.1499e-01 & 2.4925e-01 & 6.3956e-01 & 8.2246e-01 \\
    \bottomrule
    \end{tabular}
\caption{The numerical error for 2D problems on $\epsilon_u$ with different dataset.}
\label{table:overfit2D}
\end{table}

\begin{table}[h!]
\centering
\begin{tabular}{ lcccc }
    \toprule
         & \multicolumn{2}{c}{$\gamma = 0.2$}    & \multicolumn{2}{c}{$\gamma = 0.1$}\\
    \cmidrule(lr){2-3}\cmidrule(lr){4-5}
    Smooth   &  $\epsilon_u$ & $\epsilon_G$ &  $\epsilon_u$ & $\epsilon_G$ \\
    \midrule
    PW-OGA & \textbf{4.2071e-04} & \textbf{2.5830e-02} & \textbf{1.2106e-03} & \textbf{7.4495e-02} \\ 
    FNO & 4.3766e-02 & - &  3.6210e-02 & - \\ 
    DON & 3.9122e-02 & - & 5.8760e-02 & - \\
    \midrule
    Oscillatory singular &  $\epsilon_u$ & $\epsilon_G$ &  $\epsilon_u$ & $\epsilon_G$ \\
    \midrule
    PW-OGA & \textbf{6.2046e-03} & \textbf{2.2252e-02} & \textbf{1.7726e-02} & \textbf{6.3057e-02} \\ 
    FNO & 9.8489e-02 & - &  8.3703e-02 & - \\ 
    DON & 2.8677e-02 & - & 5.6553e-02 & - \\
    \bottomrule
    \end{tabular}
\caption{The numerical error for 3D problems on $\epsilon_u$ and $\epsilon_G$ with different dataset.}
\label{table:overfit3D}
\end{table}

\section{Conclusion}
\label{sec:conclusion}

In this paper, we propose two novel orthogonal greedy algorithm (OGA) frameworks tailored for linear operator learning and kernel estimation tasks. These OGA methods are based on a new semi-inner product, defined using finite data and kernel integrals, ensuring theoretical optimal approximation rates, even in the seminorm. Our analysis reveals that the approximation rate improves as the null space of the operator associated with the seminorm increases. To further tackle the challenges of approximating functions with sharp transitions or rapid oscillations and enhance approximation rates, we extend OGA to a point-wise kernel estimation approach (PW-OGA), which approximates a set of $d$-dimensional functions instead of a $2d$-dimensional kernel function, achieving faster convergence rates. The proposed methods are validated through numerical experiments on operator learning for linear PDEs and kernel estimation tasks across 1D to 3D cases, including problems on irregular domains. Results demonstrate that both OGA and PW-OGA achieve theoretical optimal approximation rates and consistently outperform baseline methods, delivering order-of-magnitude improvements in accuracy. Future work will focus on investigating the causes of long \textcolor{blue}{``}warm-up" stage problem in OGA and exploring multi-level domain decomposition methods to further enhance the performance of shallow neural network under the OGA framework. 

\section{Acknowledgments}
Ran Zhang was supported in part by China Natural National Science Foundation (grant No.  22341302), the National Key Research and Development Program of China (grant No. 2020YFA0713602, 2023YFA1008803), and the Key Laboratory of Symbolic Computation and Knowledge Engineering of Ministry of Education of China housed at Jilin University.

\bibliographystyle{unsrt}
\bibliography{OGAGreenFunction}

\appendix
\section{Appendix}
\label{app:baseline}
\subsection{Deep operator network baselines}

We implement the Deep operator network(DON) using the DeepXDE\cite{don-lu2021deepxde}. For all tasks, we run a mini-batch Adam optimizer iterations for a total number of 50000 iterations(batch size 64). The learning rate is set to 0.001. The architecture of DON for each task is presented in Table \ref{table:deeponet_arch}.

\begin{table}[h!]
\centering
\begin{tabular}{ lccc }
    \midrule
     & Branch net & Trunk net & Activation function  \\
    \midrule
    one-dimensional problem & [501,256,128,128,128] & [2,128,128,128,128] & ReLU \\ 
    two-dimensional problem(disk) & [833,512,256,128,128] & [4,128,128,128,128] &  ReLU \\ 
    two-dimensional problem(H-domain) & [997,512,256,128,128] & [4,128,128,128,128] &  ReLU \\ 
    three-dimensional problem & [4913,1024,512,256,128] & [6,128,128,128,128] &  ReLU \\ 
    \bottomrule
    \end{tabular}
\caption{Architectures of DeepONet for different tasks.}
\label{table:deeponet_arch}
\end{table}

\subsection{Fourier Neural Operator baselines}
We utilize the original Fourier Neural Operator(FNO) implementation\footnote{\url{https://github.com/zongyi-li/fourier_neural_operator}} as baseline. For all tasks, we run a mini-batch Adam optimizer for a total number of 500 epochs(12500 iterations with batch size 20). The learning rate is set to 0.001 and reduce to its half every 100 epochs. The architecture of FNO for each task is presented in Table \ref{table:fno_arch}. 

\begin{table}[h!]
\centering
\begin{tabular}{ lccccc }
    \midrule
     & Modes & Width & Padding & Layers & Activation function \\
    \midrule
    one-dimensional problem & 16 & 64 & 8 & 4 & ReLU \\ 
    three-dimensional problem & 8 & 12 & 6 & 4 & ReLU \\ 
    \bottomrule
    \end{tabular}
\caption{Architectures of FNO for different tasks.}
\label{table:fno_arch}
\end{table}

\subsection{Green Learning}

We choose GreenLearning(GL) implementation\footnote{\url{https://zenodo.org/records/7701683}} as baseline. For all tasks, we run a full-batch Adam optimizer for a total number of 10000 iterations. The learning rate is set to 0.001. The architecture of GL for each task is presented in Table \ref{table:gl_arch}.

\begin{table}[h!]
\centering
\begin{tabular}{ lccc }
    \midrule
     & network & Activation function  \\
    \midrule
    one-dimensional problem & [2,50,50,50,50,1] & Rational\cite{act-boulle2020rational} \\ 
    two-dimensional problem(disk) & [4,50,50,50,50,1] &  Rational \\ 
    two-dimensional problem(H-domain) & [4,50,50,50,50,1] &  Rational \\ 
    \bottomrule
    \end{tabular}
\caption{Architectures of DeepONet for different tasks.}
\label{table:gl_arch}
\end{table}



\end{document}